\documentclass[a4paper,12pt,reqno]{amsart}

\usepackage{amsmath,amsfonts,amscd,amssymb,amsthm,graphicx,bbm,stmaryrd,color}
\usepackage[utf8]{inputenc}
\usepackage{microtype}

\theoremstyle{plain}
\newtheorem{theorem}{Theorem}[section]
\newtheorem{thm}[theorem]{Theorem}

\newtheorem{cor}[theorem]{Corollary}
\newtheorem{lemma}[theorem]{Lemma}
\newtheorem{lem}[theorem]{Lemma}
\newtheorem{proposition}[theorem]{Proposition}
\newtheorem{prop}[theorem]{Proposition}
\newtheorem{definition}[theorem]{Definition}

\newtheorem{remark}[theorem]{Remark}
\newtheorem{rem}[theorem]{Remark}

\numberwithin{equation}{section}

\let\geq\geqslant                 \let\leq\leqslant
\let\ge\geqslant                  \let\le\leqslant

\def\les{{\begin{array}{c}<\\[-10pt]\scriptstyle{\frown}\end{array}}}
\def\ges{{\begin{array}{c}>\\[-10pt]\scriptstyle{\frown}\end{array}}}

\renewcommand{\lesssim}{\les}
\renewcommand{\gtrsim}{\ges}

\newcommand{\ep}{\varepsilon}

\def\ZO{\mathrm{Z}_\Omega}

\begin{document}

\title[Crossing probabilities in topological rectangles for the FK-Ising model]{Crossing probabilities in topological rectangles \\ for the critical planar FK-Ising model}
\author[Dmitry Chelkak]{Dmitry Chelkak$^\mathrm{a,b}$}
\author[Hugo Duminil-Copin]{Hugo Duminil-Copin$^\mathrm{c}$}
\author[Cl\'ement Hongler]{Cl\'ement Hongler$^\mathrm{d}$}

\date{\today}

\thanks{\textsc{${}^\mathrm{A}$ St.~Petersburg Department of Steklov Mathematical Institute (PDMI RAS). Fontanka~27, 191023 St.~Petersburg, Russia}}
\thanks{\textsc{${}^\mathrm{B}$ Chebyshev Laboratory, Department of Mathematics and Mechanics,
St.~Petersburg State University. 14th Line, 29b, 199178 St.~Petersburg, Russia}}
\thanks{\textsc{${}^\mathrm{C}$ Section de Math\'ematiques, Universit\'e de Gen\`eve. 2-4 rue du Li\`evre, Case postale~64, 1211 Gen\`eve 4, Suisse}}
\thanks{\textsc{${}^\mathrm{D}$ Department of Mathematics, Columbia University. 2990 Broadway, New York, NY 10027, USA}}

\thanks{{\it E-mail addresses:} \texttt{dchelkak@pdmi.ras.ru}, \texttt{hugo.duminil@unige.ch}, \texttt{hongler@math.columbia.edu}}

\begin{abstract}
We consider the FK-Ising model in two dimensions at criticality. We
obtain bounds on crossing probabilities of arbitrary topological
rectangles, uniform with respect to the boundary conditions, generalizing
results of \cite{DHN11} and \cite{CS12}. Our result relies on new
discrete complex analysis techniques, introduced in \cite{Che12}.

We detail some applications, in particular the computation of so-called
universal exponents, the proof of quasi-multiplicativity properties of arm probabilities, and bounds on crossing probabilities for the classical Ising model.
\end{abstract}

\maketitle



\section{Introduction}

The Ising model is one of the simplest and most fundamental models
in equilibrium statistical mechanics. It was proposed as a model
for ferromagnetism by Lenz in 1920
\cite{Lenz}, and then studied by Ising \cite{Ising}, in an attempt to provide a microscopic explanation
for the thermodynamical behavior of magnets. In 1936, Peierls \cite{Peierls}
showed that the model exhibits a phase transition at positive temperature
in dimensions two and higher. After the celebrated exact derivation
of the free energy of the two-dimensional model by Onsager in 1944
\cite{Onsager}, the Ising model became one of the most investigated
models in the study of phase transitions and in statistical mechanics.
See \cite{Nis05,Nis09} for a historical review of the theory.

Recently, spectacular progress was made towards the rigorous description
of the continuous scaling limit of 2D lattice models at critical temperature,
in particular the Ising model \cite{Smi10,CS12}, notably thanks to the
introduction of Schramm's SLE curves (see \cite{Smi06} for a review
of recent progress in this direction). In this paper, we develop tools
that improve the connection between the discrete Ising model and the continuous
objects describing its scaling limit.

Recall that the Ising model is a random assignment of $\pm1$ spins
to the vertices of a graph $G$, where the probability of a spin configuration
$\left(\sigma_{x}\right)_{x\in G}$ is proportional to $\exp\left(-\beta H\left(\sigma\right)\right)$. The parameter
$\beta>0$ is the inverse temperature and $H(\sigma)$ is the energy,
defined as $-\sum_{x\sim y}\sigma_{x}\sigma_{y}$ (the sum is over
all pairs of adjacent vertices). 
On the square grid $\mathbb{Z}^{2}$,
an order/disorded phase transition occurs at the critical parameter
value $\beta_\mathrm{crit}:=\frac{1}{2}\ln\left(\sqrt{2}+1\right)$. Interfaces at criticality were proved to converge to SLE(3) in \cite{CDHKS13}. We refer to \cite{Dum13} for a definition of the Ising model in infinite volume and a description of the phase transition.
In order to avoid confusion with the FK-Ising model defined below,
we will call the Ising model the \emph{spin}-\emph{Ising} \emph{model.}

In 1969, Fortuin and Kasteleyn \cite{FK72} introduced
a dependent bond percolation model, called \emph{FK percolation} or
\emph{random-cluster} model, that provides a powerful geometric representation
of a variety of models, among which the Ising model. The FK model
depends on two positive parameters, usually denoted by $p$ and $q$.
Given $p\in\left[0,1\right]$ and $q>0$, the FK$\left(p,q\right)$
model on a graph $G$ is a model on random subgraphs of $G$ containing all
its vertices: the probability of a configuration $\omega\subset G$
is proportional to
\[
\left(\frac{p}{1-p}\right)^{o\left(\omega\right)}q^{k\left(\omega\right)},
\]
where $o\left(\omega\right)$ is the number of edges of $\omega$
and $k\left(\omega\right)$ the number of \emph{clusters} of $\omega$
(maximal connected components of vertices). In what follows, an edge of $\omega$ is called {\em open}. An edge of $\mathbb Z^2$ which is not in $\omega$ is called {\em closed}.

We call the FK model with $q=2$ the \emph{FK-Ising model}. In this
case, the model provides a graphical representation of the spin-Ising
model, as is best seen through the so-called Edwards-Sokal coupling~\cite{ES88}: if one samples an FK-Ising
configuration on $G$, assigns a $\pm1$ spin to each cluster by an
independent fair coin toss, and gives to each vertex of $G$ the spin of
its cluster, the configuration thus obtained is a sample of the spin-Ising model on $G$ at inverse
temperature $\beta=\frac12\log (1-p)$.
Via the Edwards-Sokal coupling, the FK-Ising model describes
how the influence between the spins of the spin-Ising model propagates across the graph: conditionally
on the FK-Ising configuration, two spins of the Ising model are
equal if they belong to the same cluster and independent otherwise.

In this paper, we will work with the \emph{critical FK-Ising model},
hence the FK model with parameter values $q=2$ and $p=p_\mathrm{crit}={\sqrt{2}}/({\sqrt{2}\!+\!1}$), which corresponds to the critical parameter $\beta_\mathrm{crit}=\frac12\log(1+\sqrt 2)$ of the spin-Ising model on $\mathbb{Z}^2$. Let us mention that FK-Ising interfaces at criticality were proved to converge to SLE(16/3) in \cite{CDHKS13}.

\subsection{Main statement\label{sub:main-statement}}

We obtain uniform bounds for crossing probabilities
for the critical FK-Ising model on general topological rectangles. These bounds were originally obtained for Bernoulli percolation in the case of ``standard'' rectangles \cite{Rus78,SW78}.

Given a topological rectangle $\left(\Omega,a,b,c,d\right)$ (i.e.
a bounded simply connected subdomain of $\mathbb{Z}^{2}$ with four marked boundary
points listed counterclockwise) and boundary conditions $\xi$ (see Section \ref{sec:FK-Ising-model} for a formal definition),
denote by $\phi_{\Omega}^{\xi}$ the critical FK-Ising probability
measure on $\Omega$ with boundary conditions $\xi$ and by $\{\left(ab\right)\leftrightarrow\left(cd\right)\}$
the event that there is a \emph{crossing }between the arcs $\left(ab\right)$
and $\left(cd\right)$, i.e. that $\left(ab\right)$ and $\left(cd\right)$
are connected by a path of edges in the FK configuration $\omega$.

Let us denote by $\ell_{\Omega}\left[\left(ab\right),\left(cd\right)\right]$
the \emph{discrete extremal length} between $\left(ab\right)$ and $\left(cd\right)$ in $\Omega$ with unit conductances (see Section~\ref{sub:discrete-extremal-length}
for a precise definition). Informally speaking, this extremal length 
measures the distance between $\left(ab\right)$ and $\left(cd\right)$
from a random walk or electrical resistance point of view. It is worth noting that $\ell_{\Omega}\left[\left(ab\right),\left(cd\right)\right]$ is scale invariant and uniformly comparable to its continuous counterpart -- the classical extremal length (inverse of the modulus) of a topological rectangle, see \cite[Proposition ~6.2]{Che12}.

Our main result is the following uniform bound for FK-Ising crossing probabilities in
terms of discrete extremal length only:
\begin{thm}
\label{thm:main-thm}\label{strong RSW} For each $L>0$ there exists
$\eta=\eta(L)\in (0,1)$ such that, for any topological rectangle $\left(\Omega,a,b,c,d\right)$ and any boundary conditions $\xi$, the following is fulfilled:

\begin{quotation}

\noindent (i)\phantom{i} if $\ell_{\Omega}\left[\left(ab\right),\left(cd\right)\right] \leq L$\,, then $\mathbb{\phi}_{\Omega}^{\xi}\left[\left(ab\right)\leftrightarrow\left(cd\right)\right] \geq \eta$;

\smallskip

\noindent (ii) if $\ell_{\Omega}\left[\left(ab\right),\left(cd\right)\right] \geq L^{-1}$, then $\mathbb{\phi}_{\Omega}^{\xi}\left[\left(ab\right)\leftrightarrow\left(cd\right)\right] \leq 1-\eta$.

\end{quotation}
\end{thm}
Such  bounds on crossing probabilities, uniform with respect to the boundary
conditions, have been obtained for standard rectangles of the form $[a,b]\times[c,d]$ in \cite[Theorem~1]{DHN11}.
The limit (as the mesh size of the lattice tends to 0) of crossing probabilities  in arbitrary
domains with specific (free/wired/free/wired) boundary conditions have been derived in \cite[Theorem~6.1]{CS12}.
In Theorem~\ref{thm:main-thm}, the crossing bounds hold in \emph{arbitrary topological
rectangles with arbitrary boundary conditions}. In particular, they are independent
of the local geometry of the boundary. Roughly speaking, our result
is a generalization of \cite{DHN11} to possibly {}``rough'' discrete
domains; this is for instance needed in order to deal with domains
generated by random interfaces.

As in \cite{DHN11}, the proof relies on discrete complex analysis.
In order to connect the FK-Ising model with discrete complex analysis objects,
we invoke the discrete holomorphic observable  introduced
by Smirnov \cite{Smi10} in the context of the FK-Ising model, as well as a representation of crossing probabilities in
terms of harmonic measures introduced in \cite{CS12}.
To obtain the desired estimate, we adapt these results and use
new harmonic measure techniques from \cite{Che12}.

\subsection{Applications}

Estimates on crossing probabilities play a very important role in rigorous
statistical mechanics, in particular for planar percolation models. Noteworthy, they
constitute the key ingredient enabling the use of the following techniques:
\begin{itemize}
\item {\em Spatial decorrelation}: probabilities of certain events in disjoint ``well separated'' sets can be factorized at the expense
of uniformly controlled constants. This factorization is based on the spatial Markov property of the model (see
Section~\ref{sec:FK-Ising-model} for details) and estimates on crossing probabilities.

\item {\em Regularity estimates and precompactness}: the uniform bounds for crossing probabilities
are instrumental to pass to the scaling limit. Namely, these bounds imply regularity estimates on the discrete random curves arising in the
model.

\item {\em Couplings of discrete and continuous interfaces}: it is useful to couple
the critical FK-Ising interfaces and their scaling limit $\mathrm{SLE}(16/3)$
so that they are close to each other (for instance whenever
the $\mathrm{SLE}(16/3)$ curve hits the boundary of the domain, so does the discrete interface with high probability). Such couplings
are in particular useful in order to obtain the full scaling limit of discrete interfaces \cite{CN06,KS12}.

\item {\em Discretization of continuous results}: thanks to uniform estimates,
one can relate the finite-scale properties of discrete models to their continuous
limits, and transfer results from the latter to the former. Thus, the so-called
arm exponents for the critical FK-Ising model can be related to the
$\mathrm{SLE}(16/3)$ arms exponents, which in turn can be computed using stochastic
calculus techniques.

\end{itemize}
While the RSW-type bounds of \cite{DHN11} already allow for a number of
interesting applications (see for instance \cite{CN09,LS12,CGN12,DGP11}),
the stronger version of such estimates provided by Theorem
\ref{thm:main-thm} increases the scope of applications. In particular, we get several new consequences that are described below in more details.

\begin{definition} In the rest of this paper, for two real-valued quantities $X$ and $Y$ depending on a certain number of parameters, we will write $X\lesssim Y$
if there exists an absolute constant $c>0$ such that $X\leq cY$ and $X\asymp Y$
if $X\lesssim Y$ and $Y\lesssim X$ at the same time.
\end{definition}

Define $\Lambda_N:=[-N,N]^2\subset\mathbb{Z}^2$. Dual edges are edges of the dual lattice $(\mathbb Z^2)^*$, a dual edge is called dual-open/dual-closed if the corresponding edge of $\mathbb{Z}^2$ that it intersects in its middle is closed/open, respectively.

We say that a path is {\em of type $1$} if it is composed of primal edges that are all open. We say that a path is {\em of type 0} if it is composed of dual edges that are all dual-open. When fixing $n<N$ and an annulus $\Lambda_N\setminus\Lambda_n$, a self-avoiding path of type 0 or 1 connecting the inner to the outer boundary of the annulus is called an {\em arm}.

Given $n<N$ and $\sigma=\sigma_1\ldots\sigma_j\in\{0,1\}^j$\,, define $A_\sigma(n,N)$ to be the event that
there are $j$ \emph{disjoint} arms $\gamma_k$ from the inner to the outer boundary of $\Lambda_N\setminus\Lambda_n$ which are of types $\sigma_k$, $1\leq k\leq j$, where the arms are indexed in counterclockwise order. E.g., $A_1(n,N)$ denotes the
event that there exists an open path from the inner to the outer boundary of $\Lambda_N\setminus\Lambda_n$.

The following theorem is crucial in the understanding of arm exponents. The proof follows ideas going back to Kesten \cite{Kes87}. Importantly, it heavily relies on Theorem~\ref{thm:main-thm} and we do not know how to derive it from previously known results on crossing probabilities.

\smallskip

Let $\phi_{\mathbb Z^2}$ denotes the unique infinite-volume FK-Ising measure at criticality.

\begin{thm}[{\bf Quasi-multiplicativity}]\label{quasi-multiplicativity}
  Fix a sequence $\sigma$. For all
  $n_1<n_2<n_3$,   $$\phi_{\mathbb Z^2}\big[A_{\sigma}(n_1,n_3)\big] \asymp \phi_{\mathbb Z^2} \big[ A_{\sigma}
  (n_1,n_2) \big]\,\phi_{\mathbb Z^2}\big[A_{\sigma}(n_2,n_3)\big],$$
  where the constants in $\asymp$ depend on $\sigma$ only.
\end{thm}

Below we mention two classical corollaries of Theorem~\ref{quasi-multiplicativity}. Let $I=(I_k)_{1\leq k\leq j}$ be a collection of disjoint intervals on the boundary of
the square $Q=[-1,1]^2$, found in the counterclockwise order on $\partial Q$. For a sequence $\sigma$ of length $j$, let $A^I_{\sigma}(n,N)$ be the event that $A_\sigma(n,N)$ occurs and the arms $\gamma_k$, $1\leq k\leq j$, can be chosen so that each $\gamma_k$ ends on $NI_k$. 

\begin{cor}\label{loc}
  Fix a sequence $\sigma$ of length $j$. For each choice of $I=(I_k)_{1\leq k\leq j}$ and for all $n<N$ such that the event $A^I_{\sigma}(n,N)$ is non-empty, one has
  $$\phi_{\mathbb Z^2}\big[A^I_{\sigma}(n,N)\big]~\asymp~\phi_{\mathbb Z^2}\big[A_{\sigma}(n,N)\big],$$
  where the constants in $\asymp$ depend on $\sigma$ and $I$ only.
\end{cor}

This leads to the computation of universal arm exponents describing the probabilities of the five-arm event in the full plane, and two- and three-arm events in the half-plane.
\begin{cor}[{\bf Universal exponents}]\label{universal exponent}
  For all $n<N$, the following is fulfilled:
  \begin{align*}
    \phi_{\mathbb Z^2}\big[A_{10110}(n,N)\big] &\asymp \left(n/N\right)^{2}, \quad
    \phi_{\mathbb Z^2}\big[A_{10}^{\mathrm{hp}}(n,N)\big]\asymp n/N, \quad
    \phi_{\mathbb Z^2}\big[A_{101}^{\mathrm{hp}}(n,N)\big]\asymp\left(n/N\right)^{2},
  \end{align*}
  where the event $A_{\sigma}^{\mathrm{hp}}(n,N)$ is the existence of $j$ disjoint $\sigma_i$-connected crossings in the half-annulus  $(\Lambda_N\setminus \Lambda_n)\cap (\mathbb{Z}\times\mathbb{Z}_+)$ and the constants in $\asymp$ are universal.
\end{cor}

\begin{rem}
It is a standard consequence of the five arms exponent computation that $\phi_{\mathbb Z^2}\big[A_{101010}(n,N)\big] \les (n/N)^{2+\alpha}$ for some $\alpha>0$ and for all $n<N$. This bound is useful in the proof of a priori regularity estimates for discrete interfaces arising in the critical FK-Ising model and their convergence to $\mathrm{SLE}(16/3)$ curves, see~\cite{AB99,KS12,CDHKS13}.
\end{rem}

The last application presented in our paper deals with crossing probabilities in the spin-Ising model. For free boundary conditions, their conformal invariance was investigated numerically in \cite{LPSA}. For alternating ``$+1$/$-1$/$+1$/$-1$'' boundary conditions, an explicit formula for the scaling limit of crossing probabilities was predicted in~\cite{BBK05} and rigorously proved in~\cite{Izy11} using SLE techniques and a priori bounds presented below. 
For the spin model, one cannot hope to obtain estimates that are completely uniform with respect to the boundary conditions since
the probability of crossing of $+1$ spins with $-1$ boundary conditions tends to $0$ in the scaling limit (this can be seen using SLE techniques). Nevertheless, it is possible to get nontrivial bounds that are sufficient to deal with regularity of spin-Ising interfaces, notably in presence of free boundary conditions. 
\begin{cor}
\label{cor:spin-ising-crossing-bounds} For each $L>0$ there exists $\eta=\eta(L)>0$ such that the following holds:
for any topological rectangle $\left(\Omega,a,b,c,d\right)$ with
$\ell_{\Omega}\left[\left(ab\right),\left(cd\right)\right]\leq L$,
\[
\mathbb{P}\left[\mbox{there is a crossing of $-1$ spins connecting $\left(ab\right)$ and $\left(cd\right)$}\right]\geq \eta,
\]
where $\mathbb P$ denotes the critical spin-Ising model on $\left(\Omega,a,b,c,d\right)$ with free boundary conditions on $\left(ab\right)\cup\left(cd\right)$
and $+1$ boundary conditions on $\left(bc\right)\cup\left(da\right)$.
\end{cor}

By monotonicity of the spin-Ising model with respect to the boundary conditions (this is an easy consequence of the FKG inequality), Corollary~\ref{cor:spin-ising-crossing-bounds} remains fulfilled for

\begin{itemize}

\item free boundary conditions everywhere on the boundary of $\Omega$;

\item $-1$ boundary conditions on $\left(ab\right)\cup\left(cd\right)$ and $+1$ ones on $\left(bc\right)\cup\left(da\right)$.

\end{itemize}

\begin{remark} Both setups are symmetric with respect to the global spin-flip $+1$/$-1$. For topological reasons, there cannot be simultaneously a~$-1$ crossing from $\left(ab\right)$ to $\left(cd\right)$ and a~$+1$ crossing from $\left(bc\right)$ to $\left(da\right)$ even if we admit two consecutive spins to share a face instead of an edge for one of these crossings.
Due to the uniform estimate $\ell_{\Omega}\left[\left(ab\right),\left(cd\right)\right]\cdot \ell_{\Omega}\left[\left(bc\right),\left(da\right)\right]\asymp 1$ (see Section~\ref{sub:discrete-extremal-length}),
{such crossing probabilities in the critical spin-Ising model are also uniformly bounded from above if $\ell_{\Omega}\left[\left(ab\right),\left(cd\right)\right]\geq L^{-1}$}.\end{remark}

\subsection*{Acknowledgements}
The authors would like to thank P.~Nolin for many interesting discussions. The authors are also grateful to S.~Smirnov for introducing them to the subject and sharing many ideas.
D.~C. was partly supported by the Chebyshev Laboratory at Saint Petersburg State University under the Russian Federation Government grant 11.G34.31.0026 and JSC~``Gazprom Neft''.
H.~D.-C. was partly supported by the Swiss NSF and ERC AG CONFRA. C.~H. was partly supported by the National Science Foundation under grant DMS-1106588 and the Minerva Foundation.

\section{\label{sec:graph-fk-notation}FK-Ising model on discrete domains}
\setcounter{equation}{0}

\subsection{Discrete domains\label{sub:graph}}

Most of the time, a finite planar graph $G\subset\mathbb{Z}^2$  will be identified with the set of
its vertices. We will also denote by $\mathcal{E}\left(G\right)$
the set of its edges. For two vertices
$x,y\in \mathbb{Z}^2$, we write $x\sim y$ if they are adjacent and we denote
by $xy\in\mathcal{E}\left(\mathbb{Z}^2\right)$ the edge between them. In this paper, we always assume that $G$ is connected and simply connected meaning that all edges surrounded by a cycle from $\mathcal{E}(G)$ also belong to $\mathcal{E}(G)$. We call such graphs \emph{discrete domains}. For a discrete domain $\Omega$, introduce the vertex boundary of $\Omega$:
\begin{align*}
\partial\Omega&:=\{x\in \Omega:\exists y\in \mathbb{Z}^2:x\sim y~\text{and}~xy\not\in\mathcal{E}(\Omega)\}.
\end{align*}
As $\Omega$ is simply connected, there exists a natural cyclic order on $\partial\Omega$.
For $x,y\in\partial\Omega$, we denote by $\left(xy\right)\subset\partial\Omega$
the counterclockwise arc of $\partial\Omega$ from $x$ to $y$ including $x$ and $y$.
We will also frequently identify $x\in\partial\Omega$ with the arc $\left(xx\right)$.
We call a discrete domain $\Omega$ with four marked vertices $a,b,c,d\in\partial\Omega$ listed counterclockwise a \emph{topological rectangle.}

\begin{figure}
\begin{center} \includegraphics[width=0.70\textwidth]{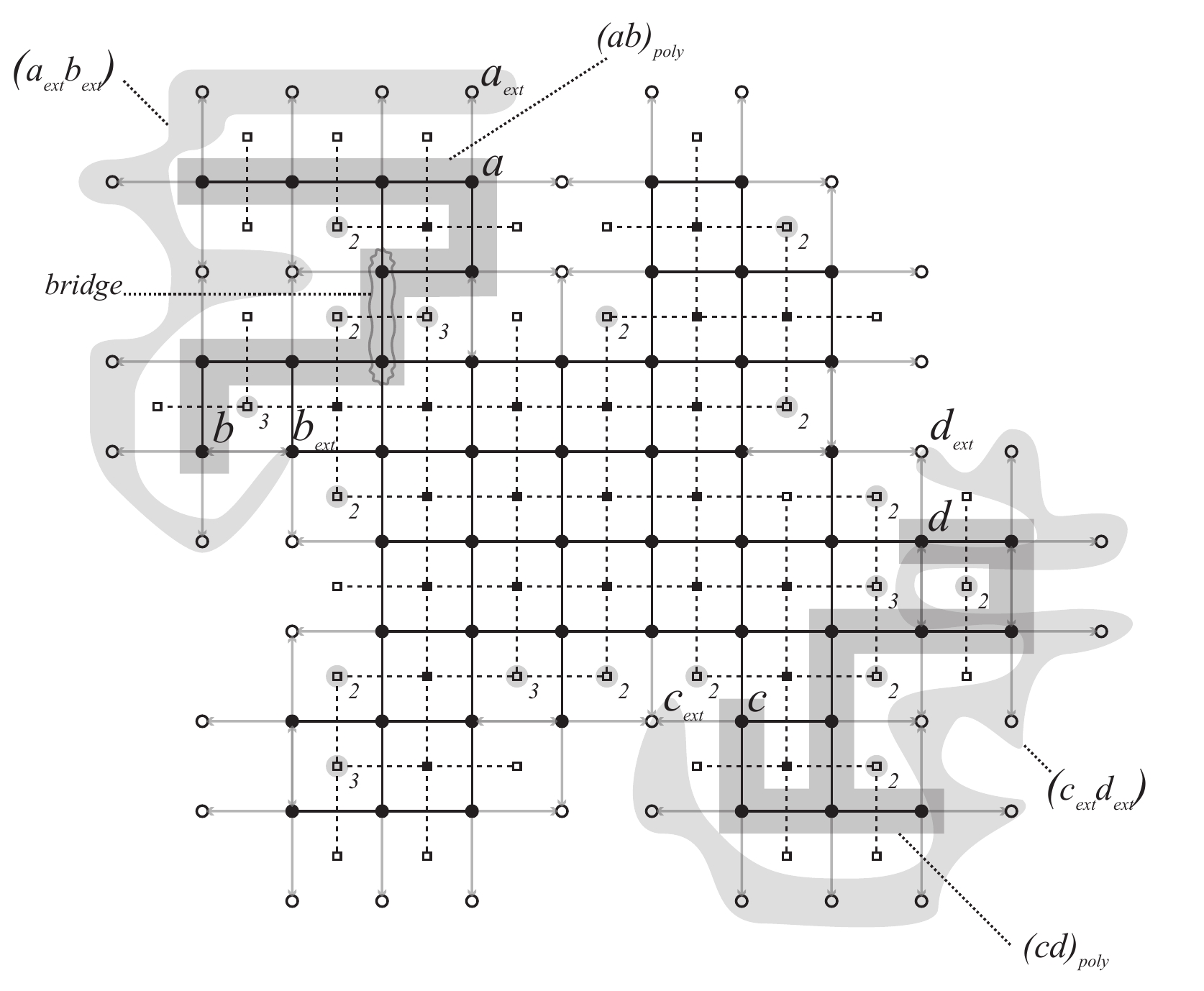}\end{center}
\caption{\label{fig:graph} An example of a discrete domain $\Omega$ (black discs). The black solid edges are elements of $\mathcal{E}(\Omega)$, the gray (oriented) edges are elements of $\mathcal{E}_\mathrm{ext}(\Omega)$. The vertices of the dual domain $\Omega^*_\mathrm{int}$ are shown as black squares. The external vertices of $\Omega$ and $\Omega^*_\mathrm{int}$ (counted with multiplicities) are shown in white. If $\Omega$ contains bridges (i.e., edges that cannot be deleted keeping $\Omega$ connected), then $\Omega^*$ is not connected.
For $a,b,c,d\in\partial\Omega$, the corresponding external boundary arcs $(a_\mathrm{ext}b_\mathrm{ext}),(c_\mathrm{ext}d_\mathrm{ext})\subset\partial_\mathrm{ext}\Omega$ are shown in gray. Also, the ``internal polyline realizations'' of boundary arcs $(ab),(cd)\subset\partial\Omega$ which are used in the proof of Proposition~\ref{prop:bound-cp-wfwf} are highlighted. Note that $(ab)_\mathrm{poly}$ and $(cd)_\mathrm{poly}$ contain inner vertices of $\Omega$.}
\end{figure}

\subsection{FK percolation models\label{sec:FK-Ising-model}\label{sec:input}}
In order to remain as self-contained as possible, some basic features
of the FK percolation (or random-cluster) models are presented now. The reader can consult
the reference book \cite{G_book_FK} for additional details.

The \emph{FK percolation} measure on a discrete domain $\Omega$ is defined as follows. A \emph{configuration} $\omega\subset\mathcal{E}(\Omega)$ is a random subgraph of $\Omega$.
An edge is called \emph{open} if it belongs to $\omega$, and \emph{closed} otherwise. Two vertices $x,y\in \Omega$ are said to be \emph{connected} if there is an \emph{open path} (a path composed
of open edges only) connecting them. Similarly, two sets of vertices $X$ and $Y$ are said to be connected if there exist two vertices $x\in X$
and $y\in Y$ that are connected; we use the notation $X\leftrightarrow Y$ for this event. We also write $x\leftrightarrow Y$ for $\{x\}\leftrightarrow Y$.
Maximal connected components of the configuration are called \emph{clusters}.

A set of \emph{boundary conditions} $\xi=(E_{1},E_{2},\ldots)$ is a
partition of $\partial \Omega$ into disjoint subsets $E_{1},E_{2},\ldots\subset\partial \Omega$.
For conciseness, singletons subsets are omitted from the notation. We say that two
boundary vertices $x,y\in\partial \Omega$ are \emph{wired} if they belong
to the same element of~$\xi$; we call boundary vertices that are
not wired to other vertices \emph{free}.

We denote by $\omega\cup\xi$ the graph obtained from the configuration
$\omega$ by artificially linking together all pairs of vertices
$x,y\in\partial\Omega$ that are wired by $\xi$. Let $o(\omega)$ and
$c(\omega)$ denote the number of open and closed edges of $\omega$, respectively, and $k(\omega,\xi)$ be the number of connected components
of $\omega\cup\xi$. The probability measure $\phi_{p,q,\Omega}^{\xi}$
of the random-cluster model on $\Omega$ with parameters $p$ and $q$
and boundary conditions $\xi$ is defined by
\[
\phi_{p,q,\Omega}^{\xi}(\left\{ \omega\right\} ):=\frac{p^{o(\omega)}(1-p)^{c(\omega)}q^{k(\omega,\xi)}}{Z_{p,q,\Omega}^{\xi}}
\]
for every configuration $\omega$ on $\Omega$, where $Z_{p,q,G}^{\xi}$ is a normalizing constant (it is also called partition function of the model). In the following, $\phi_{p,q,\Omega}^{\xi}$ also denotes the expectation with respect to the measure.

\begin{remark}If an edge $e$ connects two boundary vertices wired by $\xi$, then the event $e\in\omega$ is independent of the rest of $\omega$ since the number of clusters $k(\omega,\xi)$ does not depend on the state of~$e$. Similarly, if $e\in\mathcal{E}(\Omega)$ is a bridge (i.e. an edge disconnecting the graph into two connected components) splitting $\Omega$ into two discrete domains $\Omega_1$ and $\Omega_2$ and if boundary conditions $\xi$ do not mix $\partial\Omega_1$ and~$\partial\Omega_2$, then $\omega\cap\mathcal{E}(\Omega_1)$, $\omega\cap\mathcal{E}(\Omega_2)$ and the state of $e$ are mutually independent.
\end{remark}

\subsection{Domain Markov property}
The \emph{domain Markov property} enables one to encode the dependence between different areas of the space through boundary conditions. Namely, for each boundary conditions $\xi$ and a configuration $\varpi\subset \mathcal{E}(\Omega)\setminus\mathcal{E}(\Omega')$ outside~$\Omega'\subset\Omega$, $\phi_{p,q,\Omega}^{\xi}$~conditioned to match $\varpi$ on $\mathcal{E}(\Omega)\setminus\mathcal{E}(\Omega')$ is equal to $\phi_{p,q,\Omega'}^{\varpi\cup\xi}$, where $\varpi\cup\xi$
is the set of connections inherited from $\varpi$: one wires all vertices of $\partial\Omega'$ that are connected by $\varpi\cup\xi$. Thus, the influence of the configuration outside $\Omega'$ and boundary conditions on $\partial\Omega$ is completely contained in the new boundary conditions on $\partial\Omega'$.

\subsection{FKG inequality and monotonicity with respect to boundary conditions}
The random-cluster model on a finite graph with parameters $p\in[0,1]$ and $q\geq1$ has the \emph{strong positive association property}, a fact which has two important consequences. The first is the well-known \emph{FKG inequality}:
\[
\phi_{p,q,\Omega}^{\xi}(A_1\cap A_2)\geq\phi_{p,q,\Omega}^{\xi}(A_1)\:\phi_{p,q,\Omega}^{\xi}(A_2)
\]
for all pairs $A_1,A_2$ of \emph{increasing events} ($A$ is an increasing event if $\omega\in A$ and $\omega\subset\omega'$ implies $\omega'\in A$) and arbitrary boundary conditions
$\xi$. 

The second consequence of the strong positive association is the following
monotonicity with respect to boundary conditions, which is particularly useful
when combined with the domain Markov property. For any pair of boundary conditions
$\xi\leq\xi'$ (which means that all vertices wired in $\xi$ are wired in $\xi'$ too) and for any increasing event $A$,
we have
\[
\phi_{p,q,\Omega}^{\xi}(A)\leq\phi_{p,q,\Omega}^{\xi'}(A).
\]
Among all possible boundary conditions, the following four play a specific role in our paper:
\begin{itemize}
\item the \emph{free} boundary conditions $\xi=\emptyset$ corresponds to the case when there are no wirings between boundary vertices;
\item the \emph{wired} boundary conditions $\xi=\partial\Omega$ corresponds to the case when all boundary vertices are pairwise connected;
\item for a discrete domain $\Omega$ with two marked boundary points $a,b\in\partial\Omega$, the boundary conditions $\xi=(ab)$ are called \emph{Dobrushin} ones (in other words, all vertices
on the boundary arc $(ab)$ are wired together, and all other boundary vertices are free);
\item for a topological rectangle $\left(\Omega,a,b,c,d\right)$, the boundary conditions $\xi=((ab),(cd))$ are called \emph{alternating} (or \emph{free/wired/free/wired}) ones.
\end{itemize}
\begin{rem}
\label{rem:monotonicity} The free and wired boundary conditions are
extremal for stochastic domination: for all boundary conditions $\xi$ and any increasing event $A$,
$
\phi_{p,q,\Omega}^{\emptyset}(A)\leq\phi_{p,q,\Omega}^{\xi}(A)\leq\phi_{p,q,\Omega}^{\partial\Omega}(A)
$. 
Hence to get a lower (respectively an upper) bound on crossing
probabilities that is uniform with respect to $\xi$,
it is enough to get such a bound for $\xi=\emptyset$ (respectively $\xi=\partial\Omega$).
\end{rem}

\subsection{Planar self-duality and dual domains\label{sub:planar-duality}}
We denote by $\left(\mathbb{Z}^{2}\right)^{*}$ the dual lattice to the original (primal) square lattice $\mathbb{Z}^2$:
vertices of $\left(\mathbb{Z}^{2}\right)^{*}$ are the centers of the faces of $\mathbb{Z}^{2}$, and edges of $\left(\mathbb{Z}^{2}\right)^{*}$ connect nearest neighbors together.

The FK-Ising model is self-dual if $p=p_\mathrm{crit}(q)=\sqrt{q}/(\sqrt{q}+1)$, see also \cite{BDC12} where it is proved that $p_\mathrm{crit}(q)$ is indeed the critical (and not only self-dual) value of the FK percolation for all $q\geq 1$. This self-duality can be described as follows:
given a discrete domain $\Omega\subset\mathbb{Z}^2$, one can couple two critical FK-Ising models defined on~$\Omega$ and on an appropriately chosen \emph{dual domain} $\Omega^{*}\subset (\mathbb{Z}^2)^*$ in such a way that, whenever an edge $e\in\mathcal{E}\left(\Omega\right)$ is open, the dual edge
$e^{*}\in\mathcal{E}\left(\Omega^{*}\right)$ is closed, and vice versa. In this coupling, one should be careful with boundary conditions of the models: informally speaking, they also should be chosen dual to each other.

Let us provide a few more details regarding the dual domain $\Omega^*$ and the duality between boundary conditions. Given a discrete domain $\Omega$, construct $\Omega^*$ as follows. Let $\mathcal{E}(\Omega^*)$ be the set of dual edges of $\left(\mathbb{Z}^{2}\right)^{*}$ corresponding to the edges of $\mathcal{E}(\Omega)$. The set of vertices of $\Omega^*$ is defined to be the set of endpoints of $\mathcal{E}(\Omega^*)$ counted with multiplicities if exactly two opposite edges incident to a dual vertex belong to $\mathcal{E}(\Omega^*)$, see Fig.~\ref{fig:graph}. Then, one can couple the critical FK-Ising model on $\Omega$ with \emph{wired} boundary conditions and the critical FK-Ising model on $\Omega^*$ with \emph{free} boundary conditions so that each primal edge is open if and only if its dual is closed. In general, it can happen that the graph $\Omega^*$ is not connected, then the critical FK-Ising model on $\Omega^*$ should be understood as the collection of mutually independent models on connected components of $\Omega^*$.

Below we also use the following notation: we call $f$ an \emph{interior} vertex of $\Omega^*$ if $f$ is the center of a face of $\Omega$. We denote by $\Omega^*_{\mathrm{int}}$ the (not necessarily connected) subgraph of $\Omega^*$ formed by all interior vertices and edges between them. It is worth noting that \emph{$\Omega^*_\mathrm{int}$ is connected if $\Omega$ ``is made of square tiles'', i.e., does not contain bridges}.

\section{\label{sec:discrete complex analysis}Discrete complex analysis}
\setcounter{equation}{0}

In this section, we introduce the discrete harmonic measures and random walk partition functions that will be used in this article. A number of their properties are provided, including factorization properties and
uniform comparability results obtained in \cite{Che12}.

In order to properly define the following notions, we will need to introduce a natural extension of the domain $\Omega$. Let
$$\mathcal E_\mathrm{ext}(\Omega):=\{\overrightarrow{xy}:x\in \partial\Omega,~y\in\mathbb{Z}^2,~x\sim y~\mathrm{and}~xy\notin \mathcal{E}(\Omega)\}.$$
We will sometimes see $\mathcal E_\mathrm{ext}(\Omega)$ as a set of vertices $\partial_{\mathrm{ext}}\Omega$ by identifying oriented edges $\overrightarrow{xy}$ with their endpoints. We treat $\partial_{\mathrm{ext}}\Omega$ as a set of abstract vertices, meaning that even if some $y\in\mathbb Z^2$ is the endpoint of two (or three) oriented edges $\overrightarrow{x_1y}$ and $\overrightarrow{x_2y}$ from $\mathcal{E}_\mathrm{ext}(\Omega)$ for $x_1\ne x_2$, it is considered as two (or three) \emph{distinct} elements of $\partial_\mathrm{ext}\Omega$. Then we can also see $\mathcal E_\mathrm{ext}(\Omega)$ as a set of unoriented edges of the form $xy$, with $x\in \Omega$ and $y\in\partial_\mathrm{ext}\Omega$, see Fig.~\ref{fig:graph}.

\begin{definition}Define $\overline \Omega$ to be the graph with vertex set given by $\Omega\cup\partial_{\mathrm{ext}}\Omega$ and edge set $\mathcal{E}(\overline{\Omega})$ given by $\mathcal E(\Omega)\cup \mathcal E_\mathrm{ext}(\Omega)$.
\end{definition}

As before, since $\Omega$ is a discrete domain, there exists a natural cyclic order on $\partial_{\mathrm{ext}}\Omega$. For $x$ and $y$ in $\partial_{\mathrm{ext}}\Omega$, we introduce the counterclockwise arc $(xy)$ between the two vertices.

\smallskip

\begin{center}
\noindent {\em We highlight that, for $x,y\in\partial_\mathrm{ext}\Omega$, the arc $(xy)$ is a part of $\partial_\mathrm{ext}\Omega=\partial\overline{\Omega}$ and not $\partial\Omega$.}
\end{center}

\subsection{Random walks and discrete harmonic measures\label{sub:random-walks}}

Let $\Omega\subset\mathbb{Z}^2$ be a discrete domain (see Section~\ref{sub:graph} for a definition), we consider a collection of positive conductances $\mathrm{w}_e$ defined on the set $\mathcal{E}(\overline\Omega)$. \emph{In this paper we always assume that }
$$\mathrm{w}_e:=\begin{cases}1&\text{ if }e\in \mathcal{E}(\Omega),\\
2(\sqrt{2}\!-\!1)&\text{ if }e\in\mathcal{E}_\mathrm{ext}(\Omega).\end{cases}$$
This particular choice of boundary conductances will be important in Section~\ref{sub:from-FK-to-RW}. For a function $f:\overline\Omega\to\mathbb{R}$, we define the Laplacian $\Delta_\Omega f$ by
\[
[\Delta_\Omega f](x) := \mathrm{m}_x^{-1}\sum_{y\sim x}\mathrm{w}_{xy}(f(y)-f(x)),\]
where
\[\mathrm{m}_x:=\begin{cases}\sum_{y\sim x}\mathrm{w}_{xy}&x\in\Omega,\\
2\sqrt 2\!+\!1&x\in \partial_\mathrm{ext}\Omega.\end{cases}
\]
The notation $\mathrm{m}_x=2(\sqrt{2}\!-\!1)+3$ for $x\in \partial_\mathrm{ext}\Omega$ is introduced to fit definitions in~\cite{Che12}.

\begin{remark}\label{rem:dual-conductances}
In Section~\ref{sec:proof} we will also need to work with a dual domain $\Omega^*_\mathrm{int}$ and its extension $\overline{\Omega^*_\mathrm{int}}$ provided that $\Omega$ does not contain bridges. In this case, the only distinction between $\Omega^*$ and~$\overline{\Omega^*_\mathrm{int}}$ is that the boundary vertices of~$\Omega^*$ are ``counted with multiplicities'' in~$\overline{\Omega^*_\mathrm{int}}$. On the dual lattice, we set
$\mathrm{w}_e:=1$ for every $e\in\mathcal{E}(\Omega^*)$. All estimates from~\cite{Che12} mentioned below are uniform with respect to the choice of edge conductances as soon as there exists an absolute constant $\nu_0\ge 1$ such that $\mathrm{w}_e\in[\nu_0^{-1},\nu_0]$ for all edges.
\end{remark}

For $x,y\in\overline\Omega$, let $S_\Omega(x,y)$ denote the set of nearest-neighbor paths $x=\gamma_0\sim\gamma_1\sim\ldots\sim\gamma_{n}=y$ such that $\gamma_k\in\Omega$ for all $k=1,\ldots,n\!-\!1$, where $n=n(\gamma)$ is the length of $\gamma$. This set corresponds to the possible realizations of random walks (RW) from $x$ to $y$ staying in $\Omega$ (the first and/or last vertices can possibly be on $\partial_{\rm ext}\Omega$, in this case $\gamma_0\sim\gamma_1$, $\gamma_{n-1}\sim\gamma_n$ should be understood as $\overrightarrow{\gamma_1\gamma_0},\overrightarrow{\gamma_{n-1}\gamma_n}\in\mathcal{E}_\mathrm{ext}(\Omega)$, respectively). Let $\ZO\left[x,y\right]$ be the RW partition function defined by
\[
\ZO\left[x,y\right]~:= \sum_{\gamma\in S_\Omega(x,y)} \mathrm{m}_{y}^{-1}\!\prod_{k=0}^{n(\gamma)-1}\frac{\mathrm{w}_{\gamma_{k}\gamma_{k+1}}}{\mathrm{m}_{\gamma_k}}\,.
\]
For $X,Y\subset\overline\Omega$, define $\displaystyle
\ZO\left[X,Y\right]~:=\sum_{x\in X,\,y\in Y}\ZO[x,y].$ Also, $Z[x,Y]$ means $Z[\{x\},Y]$.

\begin{remark}\label{rem:hitting-proba}Let $x\in\Omega$ and $E\subset\partial_\mathrm{ext}\Omega$ be a boundary arc. We find that
\begin{equation}\label{eq:Z=RWhittingP}\begin{array}{rl}
\ZO[x,E]=(2\sqrt{2}\!+\!1)^{-1}\cdot \mathbb{P}\,[\!\!\!&\mathrm{RW}~\mathit{with~generator}~\Delta_\Omega \\ & \mathit{starting~from}~x ~
 \mathit{hits}~E~\mathit{before}~\partial_{\mathrm{ext}}\Omega\setminus E\,]\,.
\end{array}\end{equation}
In other words, up to the multiplicative constant, $\ZO[\,\cdot\,,E]$ is the \emph{discrete harmonic measure} of the set $E$ viewed from $x\in\Omega$. At the same time, it has nonzero boundary values on~$\partial_\mathrm{ext}\Omega$ since
\begin{equation}\label{ZXext=ZX}
\ZO[x_\mathrm{ext},E]=\frac{\mathrm{m}_{x_\mathrm{ext}}\mathrm{w}_{xx_\mathrm{ext}}}{\mathrm{m}_x}\cdot  \ZO[x,E]\ \ \mbox{, where}\ \overrightarrow{xx_\mathrm{ext}}\in\mathcal{E}_\mathrm{ext}(\Omega).
\end{equation}
This definition is useful in order to have a symmetric notation for $\ZO[x,y]=\ZO[y,x]$. Up to a multiplicative constant, $\ZO[x,E]$ does not depend on the conductances of the edges $yy_\mathrm{ext}$, $y_\mathrm{ext}\in E$. At the same time, varying conductances of other external edges one can change $\ZO[x,E]$ drastically, e.g., if $x$ and $E$ are connected in $\Omega$ through a long thin passage.
\end{remark}
\smallskip 

In our paper we use some factorization properties of the RW partition function $\ZO$. While in the continuum results of this kind are almost trivial (for instance, one can use conformal invariance and explicit expressions in a reference domain), it requires a rather delicate analysis to obtain uniform versions of them staying on the discrete level.

\begin{thm}[{\bf \cite[Theorem~3.5]{Che12}}]
\label{thm:factorization-triangle} Let $\Omega$ be a discrete domain with three vertices $a$, $c$, $d$ in $\partial_\mathrm{ext}\Omega$ listed counterclockwise. Then
\begin{equation}\label{eq:Zabc-factorization}
\ZO\left[a,\left(cd\right)\right]\asymp\sqrt{\frac{\ZO\left[a,c\right]\ZO\left[a,d\right]}{\ZO\left[c,d\right]}},
\end{equation}
where constants in $\asymp$ are independent of the domain.
\end{thm}

Theorem~\ref{thm:factorization-triangle} and the monotonicity of the ratio $\ZO[\,\cdot\,,c]/\ZO[\,\cdot\,,d]$ along the boundary arc $(dc)\subset\partial_\mathrm{ext}\Omega$ (e.g., see \cite[Section~4.1]{Che12}) imply the following estimate for the partition function $\ZO[(ab),(cd)]$ of random walks in topological rectangles.

\begin{cor}[{\bf \cite[Proposition~4.7]{Che12}}]
\label{thm:unif-comp-cross-ratio-part-fct} Let $\Omega$ be a discrete domain with four vertices $a,b,c,d$ in $\partial_\mathrm{ext}\Omega$ listed counterclockwise. Then
\begin{equation}\label{eq:Zabcd-estimate}
\ZO[(ab),(cd)]\gtrsim \sqrt{\frac{\ZO[a,c]\ZO[b,d]}{\ZO[a,b]\ZO[c,d]}},
\end{equation}
where the constant in $\gtrsim$ is independent of the domain.
\end{cor}
\begin{remark}If boundary arcs $(ab)$ and $(cd)$ are ``not too close to each other'', the one-sided estimate of $\ZO[(ab),(cd)]$ given above can be replaced by $\asymp$, see \cite[Eq.~(4.1),(4.3)]{Che12}, but we do not need this sharper result in our paper.\end{remark}

\subsection{Separators\label{sub:separators}}

A crucial concept in the following study is the notion of separators. They will allow us to perform an efficient surgery of discrete domains. Informally speaking, a separator between two marked boundary arcs $A$ and $B$ of a discrete domain $\Omega$ is a cross-cut which splits $\Omega$ into two pieces $\Omega_A\supset A$ and $\Omega_B\supset B$ in a ``good'' manner from the harmonic measure point of view. In principle, there are several possible ways to choose such cross-cuts, below we use the construction from \cite{Che12}.

Given a discrete domain $\Omega$ with four vertices $a_1,a_2,b_1,b_2\in\partial_\mathrm{ext}\Omega$ listed counterclockwise and a real parameter $k>0$, denote
\[\begin{array}{rcl}
\Omega_A=\Omega_A^B[k]&\!\!:=\!\!&\{u\in\Omega: \ZO[u;A]\geq k\ZO[u;B]\}\,,\\ \Omega_B=\Omega_B^A(k^{-1})&\!\!:=\!\!&\{u\in\Omega: \ZO[u;A]< k\ZO[u;B]\}\,,
\end{array}\]
where $A=(a_1a_2)$ and $B=(b_1b_2)$. Let
\[
L_k:=\{xy\in\mathcal{E}(\Omega):x\in\Omega_A^B[k],~y\in\Omega_B^A(k^{-1})\}.
\]
We call $L_k$ a \emph{discrete cross-cut} separating $A$ and $B$  in $\Omega$ if both $\Omega_A^B[k]$ and $\Omega_B^A(k^{-1})$ are nonempty and connected (this can fail, e.g., if there exist two edges ${xa},{xb}\in\mathcal{E}_\mathrm{ext}(\Omega)$ with $a\in A$ and $b\in B$ or if $k$ is chosen inappropriately so that one of the sets $\Omega_A$ and $\Omega_B$ is ``too thin''). The set $L_k$ can be understood as a part of $\partial_{\rm ext}\Omega_A$ as well as a part of $\partial_{\rm ext}\Omega_B$.


\begin{thm}[{\bf \cite[Theorem~5.1]{Che12}}]
\label{thm:existence-separators} Let $\Omega$, $A$, $B$, $k$, $\Omega_A$, $\Omega_B$ and $L_k$ be as above.

\smallskip

\noindent (i) For each $K\geq 1$, if $\ZO[A,B]\leq K$ and $K^{-1}\leq k\leq K$, then
\begin{equation} \label{separators-properties}\begin{array}{rl}
\mathrm{Z}_{\Omega_A}[A,L_k] \cdot \mathrm{Z}_{\Omega_B}[L_k,B] & \!\!\asymp \ \ZO[A,B],\\
\mathrm{Z}_{\Omega_A}[A,L_k] \,/\, \mathrm{Z}_{\Omega_B}[L_k,B] & \!\!\asymp \ k,
\end{array}\end{equation}
where constants in $\asymp$ may depend on $K$ but are independent of $\Omega$, $A$, $B$ and $k$.

\smallskip

\noindent (ii) There exists a constant $\kappa_0>0$ such that if $\ZO[A,B]\leq \kappa_0$ and $\kappa_0^{-1}\ZO[A,B]\leq k\leq \kappa_0(\ZO[A,B])^{-1}$, then both estimates (\ref{separators-properties}) are fulfilled with some absolute constants. Moreover, in this case $\Omega_A$ and $\Omega_B$ are always connected.
\end{thm}

Let us give a corollary which will be particularly useful for us:
\begin{cor}
\label{cor:small-separator} Let $\Omega$ be a discrete domain with four vertices $a,b,c,d$ in $\partial_\mathrm{ext}\Omega$ listed counterclockwise. Set $A=(ab)$ and $B=(cd)$. There exist two absolute constants $\zeta_0,\ep_0\in (0,1)$ such that the following holds.
If $\ZO[A,B]\leq \zeta_0$ and a real number $\zeta$ is chosen so that \mbox{$\zeta_0^{-1}\ZO[A,B]\leq \zeta\leq 1$}, then one can find $k=k(\zeta)$ such that
$L=L_k$ is a discrete cross-cut separating $A$ and $B$ in $\Omega$ with
\begin{equation}\label{eq:separators-Zab-factorization}
\ZO[A,B]\asymp \mathrm{Z}_{\Omega_A}[A,L] \cdot \mathrm{Z}_{\Omega_B}[L,B].
\end{equation}
Above, constants in $\asymp$ are independent of $(\Omega,a,b,c,d)$, and $\ep_0\zeta\leq \mathrm{Z}_{\Omega_A}[A,L]\leq \zeta$.
\end{cor}
\begin{proof} As soon as $\zeta_0\leq \kappa_0$ and $\kappa_0^{-1}\ZO[A,B]\leq k\leq \kappa_0(\ZO[A,B])^{-1}$, Theorem~\ref{thm:existence-separators}(ii) guarantees that $L_k$ is a discrete cross-cut separating $A$ and $B$ in $\Omega$ such that the estimates (\ref{separators-properties}) are fulfilled with some absolute constants. In particular, in this case we have
\[
\ZO[A,B]\asymp k^{-1}\cdot (\mathrm{Z}_{\Omega_A}[A,L_k])^2.
\]
i.e., there exist two absolute constants $c_1,c_2>0$ such that
\[
c_1\sqrt{k\ZO[A,B]}\leq\mathrm{Z}_{\Omega_A}[A,L_k]\leq c_2\sqrt{k\ZO[A,B]}\,.
\]
Without loss of generality, we may assume that $c_1^2\leq \kappa_0^{-1}\leq c_2^2$.
Let $\ep_0:={c_1/c_2}$, $\zeta_0:=\ep_0\kappa_0$ and choose $k(\zeta):=\zeta^2/(c_2^2\ZO[A,B])$. It is easy to check that our assumptions on $\zeta$ imply $\kappa_0^{-1}\ZO[A,B]\leq k(\zeta)\leq \kappa_0(\ZO[A,B])^{-1}$, as needed.
\end{proof}

\subsection{\label{sub:discrete-extremal-length} Discrete extremal length} A very useful tool when dealing with discrete harmonic measures in topological
rectangles is a discrete version of the classical extremal length. Recall that topological rectangle is a discrete domain $\Omega$ with four marked boundary points $a,b,c,d$ on $\partial\Omega$ (and not on $\partial_\mathrm{ext}\Omega$ as it was in Sections~\ref{sub:random-walks},~\ref{sub:separators}) listed counterclockwise. Given $(\Omega,a,b,c,d)$, let $\ell_\Omega[(ab),(cd)]$ denote the resistance of the electrical network $\Omega$ (with unit conductances on all edges $e\in\mathcal{E}(\Omega)$) between $(ab)$ and $(cd)$. Equivalently, one can define $\ell_\Omega[(ab),(cd)]$ as the solution to the following extremal problem:
\begin{equation}\label{ell definition}
\ell_\Omega[(ab),(cd)]:=\sup_{g:\mathcal{E}(\Omega)\to \mathbb{R}_+}\frac{[\inf_{\gamma:(ab)\leftrightarrow(cd)}\sum_{e\in\gamma}g_e]^2}{\sum_{e\in\mathcal{E}(\Omega)}g_e^2},
\end{equation}
where the infimum is taken over all nearest-neighbor paths $\gamma$ connecting $(ab)$ and $(cd)$,
see \cite[Section~6]{Che12} for details. It is important that the discrete extremal length measures the distance between $(ab)$ and $(cd)$ in a particularly robust manner as it is discussed below. 

In order to make the statements precise we need an additional notation. Given $x\in\partial\Omega$, let $x_\mathrm{ext}\in\partial_\mathrm{ext}\Omega$ be the corresponding external vertex (if there are several external edges incident to $x$, we fix $xx_\mathrm{ext}$ to be the last of them when tracking $\partial_\mathrm{ext}\Omega$ counterclockwise).
 Thus, $xx_\mathrm{ext}\in\mathcal{E}_\mathrm{ext}(\Omega)$ and, by definition, this is the only edge of $\overline{\Omega}$ incident to $x_\mathrm{ext}$. Further, let $xx_\mathrm{ext}x''x'$ be a face of $\mathbb{Z}^2$ to the left of $\overrightarrow{xx_\mathrm{ext}}$ and $x^*$ denote the center of this face.  Provided that $\Omega$ does not contain bridges, we have $xx'\in\mathcal{E}(\Omega)$ and $x^*\in\Omega^*\setminus\Omega^*_\mathrm{int}$. Moreover one can naturally identify $x^*$ with the external vertex of $\Omega^*_\mathrm{int}$ associated to the dual edge $(xx')^*$.

For a topological quadrilateral $(\Omega,a,b,c,d)$, let $\ell_{\overline{\Omega}}[(a_\mathrm{ext}b_\mathrm{ext}),(c_\mathrm{ext}d_\mathrm{ext})]$ denote the resistance between the corresponding external boundary arcs in $\overline{\Omega}$. Provided that $\Omega$ does not contain bridges, let $\ell_{\Omega^*}[(a^*b^*),(c^*d^*)]$ denote the corresponding resistance in the extension of $\Omega^*_\mathrm{int}$ (in this notation we use $\Omega^*$ instead of $\overline{\Omega^*_\mathrm{int}}$ for shortness, see Remark~\ref{rem:dual-conductances}). Then
\begin{itemize}
\item $\ell_\Omega[(ab),(cd)]\leq\ell_{\overline\Omega}[(a_\mathrm{ext}b_\mathrm{ext}),(c_\mathrm{ext}d_\mathrm{ext})]\leq \ell_\Omega[(ab),(cd)]+4(2\sqrt{2}\!-\!1)$ (note that the boundary arcs $(ab),(cd)\subset\partial\Omega$ may share a vertex while it is impossible for $(a_\mathrm{ext}b_\mathrm{ext})$ and $(c_\mathrm{ext}d_\mathrm{ext})$, thus the above extremal length in $\overline\Omega$ is always strictly positive);
\item provided that $\Omega$ does not contain bridges, one has
    \begin{equation}\label{ell comparable}
    \ell_{\Omega^*}[(a^*b^*),(c^*d^*)]\asymp \ell_{\overline\Omega}[(a_\mathrm{ext}b_\mathrm{ext}),(c_\mathrm{ext}d_\mathrm{ext})]
    \end{equation}
    (note that such a general result would \emph{not} hold for the RW partition functions $\ZO$);
\item discrete extremal lengths satisfy the following self-duality property:
    \begin{equation}\label{ell self-duality}
    \ell_{\overline\Omega}[(a_\mathrm{ext}b_\mathrm{ext}),(c_\mathrm{ext}d_\mathrm{ext})]\cdot \ell_{\overline\Omega}[(b_\mathrm{ext}c_\mathrm{ext}),(d_\mathrm{ext}a_\mathrm{ext})]\asymp 1,
    \end{equation}
    where the constants in $\asymp$ do not depend on $(\Omega,a,b,c,d)$.
\end{itemize}
The property (\ref{ell comparable}) is a direct corollary of \cite[Proposition~6.2]{Che12}: both extremal lengths are uniformly comparable to their continuous counterparts which are uniformly comparable to each other (also, one can easily modify the proof given in~\cite{Che12} so that to have the same continuous approximations for both discrete extremal lengths). The property (\ref{ell self-duality}) also immediately follows from the comparison with continuous extremal lengths which are known to be inverse of each other, see~\cite[Corollary~6.3]{Che12}.

At the same time, the discrete extremal lengths allows one to control the RW partition functions in $\Omega$ with Dirichlet boundary conditions. Recall that, following~\cite{Che12}, in Section~\ref{sub:random-walks} we formally work with the external boundary $\partial_\mathrm{ext}\Omega$ and not with $\partial\Omega$ but $\ZO[(a_\mathrm{ext}b_\mathrm{ext}),(c_\mathrm{ext}d_\mathrm{ext})]\asymp \ZO[(ab),(cd)]$ with some absolute constants in $\asymp$, see (\ref{ZXext=ZX}).

\begin{theorem}[{\cite[Theorem~7.1]{Che12}}] \label{thm:ell-Z-bounds}
There exist two continuous decreasing functions $\zeta_1,\zeta_2:\mathbb{R}_+\to \mathbb{R_+}$ such that, for all topological rectangles $(\Omega,a,b,c,d)$,
\[\begin{aligned}
&\text{if}~\ell_{\overline\Omega}[(a_\mathrm{ext}b_\mathrm{ext}),(c_\mathrm{ext}d_\mathrm{ext})]\leq L,~\text{then}~\ZO[(ab),(cd)]\geq\zeta_1(L);\\
&\text{if}~\ell_{\overline\Omega}[(a_\mathrm{ext}b_\mathrm{ext}),(c_\mathrm{ext}d_\mathrm{ext})]\geq L,~\text{then}~\ZO[(ab),(cd)]\leq\zeta_2(L).
\end{aligned}\]
Moreover, $\zeta_1(L)\leq \zeta_2(L)$, $\zeta_2(L)\to 0$ as $L\to\infty$, and $\zeta_1(L)\to \infty$ as $L\to 0$.
\end{theorem}

\section{Proof of Theorem~\ref{thm:main-thm}}
\setcounter{equation}{0} \label{sec:proof}

In this section we prove Theorem~\ref{thm:main-thm} by adapting the ideas from \cite{DHN11}. The proof is organized as follows.

In Section~\ref{sub:from-FK-to-RW} we discuss the relation between critical FK-Ising crossing probabilities
with alternating (free/wired/free/wired) boundary conditions and discrete harmonic measures in $\Omega$ and~$\Omega^*_\mathrm{int}$. The main tool is the fermionic observable
introduced in \cite{Smi10} and its version from \cite{CS12} (which was used to compute the scaling limit of crossing probabilities for alternating boundary conditions). Also, we give the lower bound for the first moment of the random variable
\begin{equation}\label{definition N}
\mathbf{N}:=\sum_{u\in(ab)}\sum_{v\in(cd)}\phi_{\Omega}^{\emptyset}[u\leftrightarrow v]~\mathbb{I}_{u\leftrightarrow v}
\end{equation}
in terms of the RW partition function in the dual domain $\Omega^*_\mathrm{int}$. Here, $\mathbb{I}_E$ denotes the indicator function of the event $E$.

In Section~\ref{sub:second-moment-N} we give the upper bound for the second moment of $\mathbf{N}$ in terms of the RW partition function $\ZO$ using discrete complex analysis techniques presented in Section~\ref{sec:discrete complex analysis}.

In Section~\ref{sub:proof-main-thm} we combine these estimates and prove the first part (uniform lower bound) of Theorem~\ref{thm:main-thm}. Finally, we use self-duality arguments from Section~\ref{sub:planar-duality} in order to derive the uniform upper bound for crossing probabilities.

\subsection{From FK-Ising model to discrete harmonic measure\label{sub:from-FK-to-RW}}
Let $(\Omega,a,b,c,d)$ be a topological rectangle, i.e. a discrete domains with four marked boundary vertices on $\partial\Omega$ listed counterclockwise. We consider the critical FK-Ising model on $\Omega$ with alternating boundary conditions $\xi=((ab),(cd))$: all boundary vertices along $(ab)$ are wired, the boundary arc $(cd)$ is wired too, and two other parts of $\partial\Omega$ are free.
The following proposition provides an upper bound for the probability that two wired arcs are connected to each other.

\begin{prop}
\label{prop:bound-cp-wfwf}
For any topological rectangle $\left(\Omega,a,b,c,d\right)$ one has
\begin{equation}\label{ab-cd-upper-bound}
\phi_{\Omega}^{\left(ab\right),\left(cd\right)}\left[\left(ab\right)\leftrightarrow\left(cd\right)\right]\lesssim \sqrt{\ZO\left[\left(ab\right),\left(cd\right)\right]},
\end{equation}
where the constant in $\lesssim$ does not depend on $(\Omega,a,b,c,d)$.
\end{prop}

\begin{proof} The proof essentially uses the construction from \cite[Section~6]{CS12} which we summarize below. Let $(ab)_\mathrm{poly}$ and $(cd)_\mathrm{poly}$ denote two ``internal polyline realizations'' of boundary arcs $(ab)$ and $(cd)$: e.g., $(ab)_\mathrm{poly}\subset\Omega$ consists of all vertices $x\in (ab)$ \emph{together with all ``near to boundary'' vertices of $\Omega$ needed to connect such $x$'s along $(ab)$ remaining in~$\Omega$}, see Fig.~\ref{fig:graph}. By the FKG inequality, the probability of the event $(ab)\leftrightarrow (cd)$ increases if {edges of $(ab)_\mathrm{poly}$ and $(cd)_\mathrm{poly}$ are assumed to be open}. Some trivialities can appear for the new boundary conditions (e.g., if $(ab)_\mathrm{poly}\cap(cd)_\mathrm{poly}\ne \emptyset$) but then (\ref{ab-cd-upper-bound}) holds automatically. Define
\[
\Omega':=\Omega\setminus [(ab)_\mathrm{poly}\cup(cd)_\mathrm{poly}].
\]
Note that the external boundary of $\Omega'$ is composed by the following four arcs:
\[
\partial_\mathrm{ext}\Omega'\subset(ab)_\mathrm{poly}\cup(b_\mathrm{ext}c_\mathrm{ext})\cup (cd)_\mathrm{poly}\cup (d_\mathrm{ext}a_\mathrm{ext}).
\]
In general, $\Omega'$ can be non-connected (e.g. if $(ab)_\mathrm{poly}$ ``envelopes'' some piece of $\Omega$). In this case, we use the same notation $\Omega'$ for the relevant connected component.

For this setup, in \cite[Proof of Theorem 6.1]{CS12}, two discrete {s-holomorphic} observables are introduced, and it is shown that there exists a linear combination $F$ of them and a discrete version $H$ of $\int \mathrm{Im}[F^2dz]$ which is defined on the extension $\overline{\Omega'}$ of $\Omega'$ such that
\begin{itemize}
\item $H$ is a discrete superharmonic function in $\Omega'$ (in~\cite{CS12}, spins in the Ising model live on \emph{faces} of an isoradial graph $\Gamma$, thus our $H$ is $H|_{\Gamma^*}$ in the notation of~\cite{CS12});
\item $H=0$ on $(ab)_\mathrm{poly}$\,, $H=1$ on $(b_\mathrm{ext}c_\mathrm{ext})$ and $H=\varkappa$ on $(cd)_\mathrm{poly}\cup(d_\mathrm{ext}a_\mathrm{ext})$ (recall that we have set all conductances on $\mathcal{E}_\mathrm{ext}(\Omega)$ to be $2(\sqrt{2}\!-\!1)$ instead of $1$ which is equivalent to the ``boundary modification trick'' used in \cite{CS12});
\item $H$ has nonnegative \emph{outer} normal derivative on $(b_\mathrm{ext}c_\mathrm{ext})\cup(d_\mathrm{ext}a_\mathrm{ext})$ (in other words, for each external edge $yy_\mathrm{ext}\in\mathcal{E}_\mathrm{ext}(\Omega')$ on these arcs, one has $H(y)\leq H(y_\mathrm{ext})$);
\item the value $\varkappa$ satisfies $\mathrm{Q}\asymp \sqrt{1\!-\!\varkappa}$, where $\mathrm{Q}$ is the probability of the event that there exists a crossing from $(ab)_\mathrm{poly}$ to $(cd)_\mathrm{poly}$ in $\Omega$ (e.g., see \cite[Eq.~(6.6)]{CS12}).
\end{itemize}
Denote by $a',b',c',d'\in\partial\Omega'$ the boundary vertices of $\Omega'$ such that
\[
(a'_\mathrm{ext}b'_\mathrm{ext})=\partial_\mathrm{ext}\Omega'\cap(ab)_\mathrm{poly}\quad\mathrm{and}\quad (c'_\mathrm{ext}d'_\mathrm{ext})=\partial_\mathrm{ext}\Omega'\cap(cd)_\mathrm{poly}.
\]
Let ${yy_\mathrm{ext}}$ be an external edge of $\Omega'$ with $y_\mathrm{ext}\in (d_\mathrm{ext}a_\mathrm{ext})$. Using $H(y)\leq H(y_\mathrm{ext})=\varkappa$, subharmonicity of the function $\varkappa-H$ and Remark~\ref{rem:hitting-proba}, we conclude that
\[\begin{aligned}
0\leq (1+2\sqrt2)\cdot(\varkappa-H(y))& \leq \varkappa\cdot\mathrm{Z}_{\Omega'}[y,(a'_\mathrm{ext}b'_\mathrm{ext})] +(\varkappa\!-\!1)\cdot\mathrm{Z}_{\Omega'}[y,(b'_\mathrm{ext}c'_\mathrm{ext})] \\ & ={\mathrm{Z}_{\Omega'}[y,(a'_\mathrm{ext}b'_\mathrm{ext})]} - (1\!-\!\varkappa)\cdot \mathrm{Z}_{\Omega'}[y,(a'_\mathrm{ext}c'_\mathrm{ext})].
\end{aligned}\]

We now choose $yy_\mathrm{ext}$ to be the next external edge of $\Omega'$ after $d'd'_\mathrm{ext}$ when tracking $\partial\Omega'$ counterclockwise. Then we have $\mathrm{Z}_{\Omega'}[y,\,\cdot\,]\asymp \mathrm{Z}_{\Omega'}[d',\,\cdot\,]$ with some absolute constants. Hence,
\[\begin{aligned}
1\!-\!\varkappa & \lesssim \frac{\mathrm{Z}_{\Omega'}[d',(a'b')]}{\mathrm{Z}_{\Omega'}[d',(a'c')]}
 \asymp \sqrt{\frac{\mathrm{Z}_{\Omega'}[d',b']\mathrm{Z}_{\Omega'}[a',c']} {\mathrm{Z}_{\Omega'}[a',b']\mathrm{Z}_{\Omega'}[d',c']} }\lesssim \mathrm{Z}_{\Omega'}[(a'b'),(c'd')],
\end{aligned}\]
where we used the uniform factorization property (\ref{eq:Zabc-factorization}) of the discrete harmonic measure of boundary arcs 
in $\Omega'$ and the uniform estimate (\ref{eq:Zabcd-estimate}). Therefore, we get the following sequence of uniform estimates:
\[
\phi_{\Omega}^{\left(ab\right),\left(cd\right)}\left[\left(ab\right)\leftrightarrow\left(cd\right)\right]\leq \mathrm{Q} \asymp \sqrt {1\!-\!\varkappa} \lesssim \sqrt{\mathrm{Z}_{\Omega'}[(a'b'),(c'd')]}\lesssim \sqrt{\ZO[(ab),(cd)]}.
\]
The first inequality is due to the FKG inequality as mentioned above. The last inequality follows from the following consideration: each nearest-neighbor path connecting $(a'b')$ with $(c'd')$ in $\Omega'$ can be completed into a path connecting $(ab)$ with $(cd)$ in~$\Omega$ using a uniformly bounded number of additional edges.
\end{proof}

\begin{remark} \label{rem:wfwf-collapsing}
When $a=b$, boundary conditions become Dobrushin boundary conditions and therefore
\begin{equation}\label{a-cd-upper-bound}
\phi_{\Omega}^{\left(cd\right)}\left[a\leftrightarrow\left(cd\right)\right]\lesssim \sqrt{\ZO\left[a,\left(cd\right)\right]}
\end{equation}
which can be though of as a particular case of (\ref{ab-cd-upper-bound}). This bound can also be proved independently using the basic fermionic observable \cite{Smi10} in $(\Omega,c,d)$, see~\cite{DHN11}.
\end{remark}

Similarly to (\ref{ab-cd-upper-bound}), one can give a lower bound for crossing probabilities with alternating boundary conditions in terms of RW partition functions in the \emph{dual} domain $\Omega^*$, e.g. see \cite[Proposition~3.2]{DHN11} for the corresponding counterpart of \eqref{a-cd-upper-bound}. In our paper we need only the particular case of this estimate when both arcs $(ab)$ and $(cd)$ are collapsed to points. Below we use the notation introduced in Sections~\ref{sub:planar-duality},~\ref{sub:random-walks} and~\ref{sub:discrete-extremal-length}. Recall that, for a given $x\in\partial\Omega$, $\overrightarrow{xx_\mathrm{ext}}\in\partial_\mathrm{ext}\Omega$ denotes the ``most counterclockwise'' external edge incident to $x$, and $x^*\in\Omega^*\setminus\Omega^*_\mathrm{int}$ is the center of the face $xx_\mathrm{ext}x''x'$ lying to the left of this edge.

\begin{proposition} \label{prop:spin-spin-hm} Let $\Omega$ be a discrete domain, $a,c\in\partial\Omega$, and dual vertices $a^*,c^*$ be defined as above. If the other endpoints of dual edges $(aa')^*$ and $(cc')^*$ lie in the same connected component of $\Omega^*_\mathrm{int}$, then
\begin{equation}\label{a-c-lower-bound}
\phi^\emptyset_\Omega[a\leftrightarrow c] \gtrsim \sqrt{\mathrm{Z}_{\Omega^*_\mathrm{int}}[a^*,c^*]}\,,
\end{equation}
where $\mathrm{Z}_{\Omega^*_\mathrm{int}}$ is the RW partition function in this connected component of $\Omega^*_\mathrm{int}$\,.
\end{proposition}

\begin{proof}
This proposition is directly obtained from \cite[Proposition~3.2]{DHN11} applied to the case when the wired arc is collapsed to a single edge $(aa')$ and the simple estimate $\phi^\emptyset_\Omega[a\leftrightarrow c]\asymp\phi^{(aa')}_\Omega[(aa')\leftrightarrow c]$ (it is due to the finite-energy property of the model, see \cite{G_book_FK} for details).
\end{proof}

\begin{cor}\label{cor:first-moment}
Let a discrete $\Omega$ do not contain bridges, $a,b,c,d\in\partial\Omega$ be listed counterclockwise, and dual vertices $a^*,b^*,c^*,d^*$ be defined as above.
Then
\[
\phi_\Omega^\emptyset[\mathbf{N}] \gtrsim \mathrm{Z}_{\Omega^*_\mathrm{int}}[(a^*b^*),(c^*d^*)],
\]
where  the constant in $\gtrsim$ does not depend on $(\Omega,a,b,c,d)$.
\end{cor}

\begin{proof} Applying the uniform estimate (\ref{a-c-lower-bound}), we get
\[
\phi^\emptyset_\Omega[\mathbf{N}]\ =\sum_{u\in(ab)}\sum_{v\in(cd)} \phi_{\Omega}^{\emptyset}[u\leftrightarrow v]^2\ \gtrsim \sum_{u\in(ab)}\sum_{v\in(cd)} \mathrm{Z}_{\Omega^*_\mathrm{int}}[u^*,v^*]\ = \mathrm{Z}_{\Omega^*_\mathrm{int}}[(a^*b^*),(c^*d^*)]
\]
as 
each face $f\in (a^*b^*)$ corresponds to exactly one $u\in(ab)$ (and, similarly, for $(c^*d^*)$).
\end{proof}

\subsection{\label{sub:second-moment-N}Second moment estimate for the random variable $\mathbf{N}$}

In this section we prove the crucial second moment estimate for the random variable $\mathbf{N}$ provided that $\ZO[(ab),(cd)]$ is small enough, see Proposition~\ref{prop:second-moment} below. We need a preliminary lemma. Let the absolute constants $\zeta_0,\ep_0>0$ be fixed as in Corollary~\ref{cor:small-separator}.

\begin{lem}
\label{lem:two points} For all topological rectangles $(\Omega,a,b,c,d)$ with $\ZO[(ab),(cd)]\le\zeta_0$, the following is fulfilled:
\begin{equation}
\label{two-points bound}
\phi^{(cd)}_\Omega\big[\{a\leftrightarrow(cd)\}\cap\{b\leftrightarrow(cd)\}\big]~\lesssim~\sqrt{\frac{\ZO[a,(cd)]\ZO[b,(cd)]}{\ZO[(ab),(cd)]}},
\end{equation}
where the constant in $\lesssim$ does not depend on $(\Omega,a,b,c,d)$.
\end{lem}
\begin{proof}
Without loss of generality we can assume that $\ZO[a,(cd)]$ and $\ZO[b,(cd)]$ are both less or equal to
$\tfrac{\zeta_0}3\ZO[(ab),(cd)]$.
Indeed, assume for instance that $\ZO[a,(cd)]\ge\tfrac{\zeta_0}3\ZO[(ab),(cd)] $. In such case, (\ref{a-cd-upper-bound}) gives
\begin{align*}
\phi^{(cd)}_\Omega\big[\{a\leftrightarrow(cd)\}\cap\{b\leftrightarrow(cd)\}\big]&\leq \phi^{(cd)}_\Omega[b\leftrightarrow(cd)]\lesssim \sqrt{\ZO[b,(cd)]}\\
&\lesssim\sqrt{\frac{\ZO[a,(cd)]\ZO[b,(cd)]}{\ZO[(ab),(cd)]}}\,.
\end{align*}
The same reasoning can be applied if $\ZO[b,(cd)]\ge\tfrac{\zeta_0}{3}\ZO[(ab),(cd)]$.

Provided that both $\ZO[a,(cd)]$ and $\ZO[b,(cd)]$ are less or equal to
$\tfrac{\zeta_0}3\ZO[(ab),(cd)]$, we may apply Corollary~\ref{cor:small-separator} to the discrete domain $\Omega$, boundary arcs $A=a$ and $A=b$, respectively, $B=(cd)$ and $\zeta=\tfrac{1}{3}\ZO[(ab),(cd)]$. Indeed, both $\ZO[a,(cd)]$ and $\ZO[b,(cd)]$ are less or equal to $\ZO[(ab),(cd)]\leq\zeta_0$, and
$$\zeta_0^{-1}\cdot\max\{\,\ZO[a,(cd)],\,\ZO[b,(cd)]\,\}~\le~ \tfrac{1}{3}\ZO[(ab),(cd)]=\zeta ~\le ~ \tfrac{\zeta_0}{3}\le 1.$$

Let $\Gamma_a$ denote the corresponding discrete cross-cut separating $a$ and $(cd)$ in $\Omega$ such that
\begin{equation}
\ep_0\zeta\leq\mathrm{Z}_{\Omega'_a} [\Gamma_a,(cd)]\leq\zeta,\label{bound Gamma1}
\end{equation}
here and below $\Omega_a$ and $\Omega'_a$ denote the connected components of $\Omega\setminus\Gamma_{a}$ containing $a$ and $(cd)$, respectively.
Similarly, we construct a discrete cross-cut $\Gamma_{b}$ separating $b$ and $(cd)$ in $\Omega$ such that
\begin{equation}
\ep_0\zeta\leq\mathrm{Z}_{\Omega'_b}[\Gamma_b,(cd)]\leq\zeta,\label{bound Gamma2}
\end{equation}
and use the same notation $\Omega'_b$ and $\Omega_b$ for the corresponding connected components of $\Omega\setminus\Gamma_b$. Let $\Omega'':=\Omega'_a\cap\Omega'_b$. Note that $\Gamma_a$ and $\Gamma_b$ cannot intersect since
otherwise, $\Gamma_a\cup\Gamma_b$ would separate the whole boundary arc $(ab)$ from $(cd)$, which is impossible as
\[
\mathrm{Z}_{\Omega''}[\Gamma_{a}\cup\Gamma_{b},(cd)]\leq \mathrm{Z}_{\Omega'_a}[\Gamma_{a},(cd)] + \mathrm{Z}_{\Omega'_b}[\Gamma_{b},(cd)] \leq 2\zeta\le \tfrac{2}{3}\cdot \ZO[(ab),(cd)].
\]
We are thus facing the following topological picture: the two cross-cuts $\Gamma_{a}$ and $\Gamma_{b}$ do not intersect each other and separate $a$, $b$ and $(cd)$ in $\Omega$. Let $\Gamma_{\!(cd)}'':=\Gamma_a\cup\Gamma_b\cup ((ab)\cap \partial\Omega'')$.  The spatial Markov property and the monotonicity with respect to boundary conditions (simply wire the arcs $\Gamma_{a}\subset\partial\Omega_a$, $\Gamma_{b}\subset\partial\Omega_b$ and $\Gamma_{\!(cd)}''\subset\partial\Omega''$) enable us to apply the estimates (\ref{ab-cd-upper-bound}) and (\ref{a-cd-upper-bound}) to find
\[\begin{aligned}
\phi_{\Omega}^{(cd)}\big[\{a\leftrightarrow(cd)\}\cap\{b\leftrightarrow(cd)\}\big]~&\leq~\phi_{\Omega_{a}}^{\Gamma_{a}}[a\leftrightarrow\Gamma_{a}]\cdot \phi_{\Omega_{b}}^{\Gamma_{b}}[b\leftrightarrow\Gamma_{b}]\cdot \phi_{\Omega''}^{\Gamma_{\!(cd)}''\,,\,(cd)}[\Gamma_{\!(cd)}''\leftrightarrow(cd)]\\
&\lesssim \sqrt{\mathrm{Z}_{\Omega_a}[a,\Gamma_a]\cdot\mathrm{Z}_{\Omega_b}[b,\Gamma_b]\cdot\mathrm{Z}_{\Omega''}[\Gamma_{\!(cd)}''\,,(cd)]}.
\end{aligned}\]
It follows from the factorization property (\ref{eq:separators-Zab-factorization}) of separators and the bounds (\ref{bound Gamma1}), (\ref{bound Gamma2}) that
\[
\mathrm{Z}_{\Omega_a}[a,\Gamma_a]\asymp \frac{\ZO[a,(cd)]}{\mathrm{Z}_{\Omega'_a}[\Gamma_a,(cd)]}\asymp\frac{\ZO[a,(cd)]}{\ZO[(ab),(cd)]}\quad\text{and}\quad \mathrm{Z}_{\Omega_b}[b,\Gamma_b]\asymp\frac{\ZO[b,(cd)]}{\ZO[(ab),(cd)]}\,.
\]
Due to monotonicity of the RW partition functions with respect to the domain, we also have
\[
\mathrm{Z}_{\Omega''}[\Gamma_{\!(cd)}''\,,(cd)]\leq \mathrm{Z}_{\Omega_a}[\Gamma_a,(cd)]+\mathrm{Z}_{\Omega_b}[\Gamma_b,(cd)]+\ZO[(ab),(cd)]\leq \tfrac{5}{3}\cdot \ZO[(ab),(cd)].
\]
Putting everything together we arrive at (\ref{two-points bound}).
\end{proof}

\begin{prop}\label{prop:second-moment}
There exists an absolute constant $\zeta'_0>0$ such that, for all topological rectangles $(\Omega,a,b,c,d)$ with $\ZO[(ab),(cd)]\leq \zeta'_0$, one has
\[
\phi_\Omega^\emptyset [\mathbf{N}^2] \lesssim (\ZO[(ab),(cd)])^{\frac32},
\]
where the constant in $\lesssim$ does not depend on $(\Omega,a,b,c,d)$. 
\end{prop}

\begin{proof} If $\zeta'_0$ is chosen small enough, Theorem~\ref{thm:existence-separators}(ii) applied to $\Omega$, $A=(ab)$, $B=(cd)$ and $k=1$ guarantees that there exists a discrete cross-cut $\Gamma$ splitting $\Omega$ into two subdomains $\Omega_{(ab)}$ and $\Omega_{(cd)}$ (see Section~\ref{sub:separators}) such that
\[
\mathrm{Z}_{\Omega_{(ab)}}[(ab),\Gamma] \asymp \mathrm{Z}_{\Omega_{(cd)}}[\Gamma,(cd)] \asymp \sqrt{\ZO[(ab),(cd)]}
\]
and both $\mathrm{Z}_{\Omega_{(ab)}}[(ab),\Gamma]$ and $\mathrm{Z}_{\Omega_{(cd)}}[\Gamma,(cd)]$ are less or equal to $\zeta_0$. Note that
\[
\phi_\Omega^\emptyset [\mathbf{N}^2] =\sum_{u,v\in (ab)}\sum_{u',v'\in (cd)} \phi_{\Omega}^{\emptyset}[u\leftrightarrow v]\,\phi_{\Omega}^{\emptyset}[u'\leftrightarrow v'] \, \phi_{\Omega}^{\emptyset}\big[\{u\leftrightarrow v\}\cap\{u'\leftrightarrow v'\}\big].
\]
Wiring both sides of the cross-cut $\Gamma$ and using the monotonicity of the FK-Ising model with respect to boundary conditions, we find
\[\begin{aligned}
\phi_\Omega^\emptyset [\mathbf{N}^2] & \leq \sum_{u,v\in (ab)}\phi_{\Omega_{(ab)}}^{\Gamma}[u\leftrightarrow \Gamma]\, \phi_{\Omega_{(ab)}}^{\Gamma}[v\leftrightarrow \Gamma]
\, \phi_{\Omega_{(ab)}}^{\Gamma}\big[\{u\leftrightarrow \Gamma\}\cap\{v\leftrightarrow \Gamma\}\big]\\
& \times \sum_{u',v'\in (cd)}\phi_{\Omega_{(cd)}}^{\Gamma}[u'\leftrightarrow \Gamma]\, \phi_{\Omega_{(cd)}}^{\Gamma}[v'\leftrightarrow \Gamma]
\, \phi_{\Omega_{(cd)}}^{\Gamma}\big[\{u'\leftrightarrow \Gamma\}\cap\{v'\leftrightarrow \Gamma\}\big]\\
&=: S_{(ab)}\times S_{(cd)}\,.
\end{aligned}\]
Applying the uniform estimates (\ref{a-cd-upper-bound}) and (\ref{two-points bound}) to each term of the sum $S_{(ab)}$ we get
\[
S_{(ab)}\lesssim \sum_{u,v\in (ab)} \frac{\mathrm{Z}_{\Omega_{(ab)}}[u,\Gamma]\cdot \mathrm{Z}_{\Omega_{(ab)}}[v,\Gamma]}{\sqrt{\mathrm{Z}_{\Omega_{(ab)}}[(uv),\Gamma]}}\,,
\]
where we assume that, independently of the order of $u$ and $v$ on $(ab)$, the boundary arc $(uv)$ is chosen so that $(uv)\subset (ab)$. Recall that $\mathrm{Z}_{\Omega_{(ab)}}[(uv),\Gamma]=\sum_{w\in (uv)}\mathrm{Z}_{\Omega_{(ab)}}[w,\Gamma]$ by definition. The simple technical Lemma~\ref{lem:technical} given below allows us to conclude that
\[
S_{(ab)}\lesssim (\mathrm{Z}_{\Omega_{(ab)}}[(ab),\Gamma])^{\frac 32} \asymp (\ZO[(ab),(cd)])^{\frac 34},
\]
with some absolute constants in $\lesssim$ and $\asymp$. Similarly, $S_{(cd)}\lesssim (\ZO[(ab),(cd)])^{\frac 34}$.
\end{proof}

\begin{lemma}\label{lem:technical} Let $x_1,\ldots,x_n>0$ and $X_{km}=X_{mk}:=\sum_{s=k}^m x_s$ for $1\leq k\leq m\leq n$. Then
\[
\sum_{k,m=1}^n \frac{x_kx_m}{X_{km}^{1/2}}\ \leq\ \frac{8}{3} \cdot \biggl[\sum_{s=1}^n x_s\biggr]^{\frac{3}{2}}.
\]
\end{lemma}
\begin{proof}
Let $t_0:=0$ and $t_k:=\sum_{s=1}^k x_k$ for $k=1,\ldots, n$. It is easy to see that
\[
\sum_{k,m=1}^n  \frac{x_mx_k}{X_{km}^{1/2}}\ \leq\ \sum_{k,m=1}^n \int_{t_{k-1}}^{t_k}\int_{t_{m-1}}^{t_m}\frac{dxdy}{|x\!-\!y|^{1/2}}\ =\ \int_0^{t_n}\int_0^{t_n}\frac{dxdy}{|x\!-\!y|^{1/2}}\,.
\]
The last integral is equal to $\frac{8}{3}\cdot t_n^2$ which gives the result.
\end{proof}

\subsection{\label{sub:proof-main-thm}Proof of Theorem~\ref{thm:main-thm}}

\begin{proof}[Proof of Theorem~\ref{strong RSW}(i)] Due to the monotonicity of crossing probabilities with respect to boundary conditions (see Remark~\ref{rem:monotonicity}), it is sufficient to consider the case $\xi=\emptyset$. Moreover, without loss of generality we may assume that $\Omega$ contains no bridges. Indeed, for the critical FK-Ising model on $\Omega$ with free boundary conditions, all bridges are open independently of each other with the fixed rate $p_\mathrm{bridge}=1-p_\mathrm{crit}$, and the critical FK-Ising models on the remaining components are mutually independent. Therefore, one can remove from $\Omega$ all bridges that do not separate $(ab)$ and $(cd)$ together with the components behind them: neither $\ell_\Omega[(ab),(cd)]$ nor $\phi^\emptyset_\Omega[(ab)\leftrightarrow (cd)]$ changes. Further, it is easy to see that the number of remaining bridges (separating $(ab)$ and $(cd)$) is bounded by $\ell_\Omega[(ab),(cd)]\leq L$. Thus, if we have the desired lower bound for the crossing probabilities in all remaining components (the corresponding extremal lengths there are smaller than $L$ too), then $\phi_\Omega^\emptyset[(ab)\leftrightarrow(cd)]\geq [p_\mathrm{bridge}\cdot \eta(L)]^L$ and we are done.

\smallskip

Further, for each fixed $L_0>0$ we may assume that either $\ell_\Omega[a,c]\leq L_0$ or
\begin{equation}
\label{ell L-0 bound}
L_0\leq\ell_\Omega[(ab),(cd)]\leq \max\{L,L_0\!+\!c_0\}
\end{equation}
provided that an absolute constant $c_0>0$ is chosen large enough. Indeed, let \mbox{$\ell_\Omega[a,c]>L_0$} while $\ell_\Omega[(ab),(cd)]<L_0$. Then one can shrink the boundary arcs step by step (thus increasing the extremal length), arriving at smaller arcs $(ab')$ and $(cd')$ such that $L_0\leq\ell_\Omega[(ab'),(cd')]\leq L_0+c_0$ (note that the increment of $\ell_\Omega$ on each step is uniformly bounded). Clearly, $\phi_\Omega^\emptyset [(ab')\leftrightarrow (cd')]\leq \phi_\Omega^\emptyset [(ab)\leftrightarrow (cd)]$, thus it is enough to prove the uniform lower bound for the first event.

\smallskip

Assume that (\ref{ell L-0 bound}) holds true. If $L_0=L_0(\zeta'_0)$ is chosen large enough, then Theorem~\ref{thm:ell-Z-bounds} and the lower bound in (\ref{ell L-0 bound}) yield $\ZO[(ab),(cd)]\leq\zeta'_0$, hence we can apply Proposition~\ref{prop:second-moment}. The Cauchy-Schwarz inequality and Corollary~\ref{cor:first-moment} give
\[
\phi^\emptyset_\Omega[(ab)\leftrightarrow(cd)]~=~\phi^\emptyset_\Omega[\mathbf{N}\!>\!0]~\geq~ \frac{\phi^\emptyset_\Omega[\mathbf{N}]^2}{\phi^\emptyset_\Omega[\mathbf{N}^2]}~\gtrsim~ \frac{(\mathrm{Z}_{\Omega^*_\mathrm{int}}[(a^*b^*),(c^*d^*)])^2}{\zeta_0'^{\,3/2}}.
\]
Recall that $\Omega$ contains no bridges, thus $\Omega^*_\mathrm{int}$ is connected. The upper bound in (\ref{ell L-0 bound}) and Theorem~\ref{thm:ell-Z-bounds} imply the lower bound for the right-hand side which depends on $L$ only.

The case $\ell_\Omega[a,c]\leq L_0$ is much simpler: again, due to Theorem~\ref{thm:ell-Z-bounds}, $\mathrm{Z}_{\Omega^*_\mathrm{int}}[a^*,c^*]$ is uniformly bounded from below and the result follows from Corollary~\ref{cor:first-moment}.
\end{proof}

\begin{proof}[Proof of Theorem~\ref{thm:main-thm}(ii)] Due to Remark~\ref{rem:monotonicity}, it is sufficient to consider the fully wired boundary conditions $\xi=\partial\Omega$. Again, we may assume that $\Omega$ contains no bridges: if all bridges not separating $(ab)$ and $(cd)$ are wired together, the crossing probability increases, and if there is a bridge separating $(ab)$ and $(cd)$, then this bridge is closed with probability $p_\mathrm{crit}$, thus yielding $\phi_\Omega^{\partial\Omega}[(ab)\leftrightarrow (cd)]\le p_\mathrm{crit}$.

As before, let $b^*$ and $d^*$ be the boundary faces lying to the left of the external edges $bb_\mathrm{ext}$ and $dd_\mathrm{ext}$, and $a^*_{r},c^*_{r}$ denote the boundary faces lying to the \emph{right} of $aa_\mathrm{ext}$, $cc_\mathrm{ext}$. The planar self-duality described in Section~\ref{sub:planar-duality} implies
\[
1-\phi_\Omega^{\partial\Omega}[(ab)\leftrightarrow(cd)]=\phi_{\Omega^*}^\emptyset[(b^*c^*_r)\leftrightarrow(d^*a^*_r)]\asymp \phi_{\Omega^*}^\emptyset[(b^*c^*)\leftrightarrow(d^*a^*)].
\]
Therefore, the result follows from Theorem~\ref{thm:main-thm}(i) and the uniform estimates from Section~\ref{sub:discrete-extremal-length}:
\[
\ell_{\Omega^*}[(b^*c^*),(d^*a^*)]\asymp \ell_{\overline\Omega}[(b_\mathrm{ext}c_\mathrm{ext}),(d_\mathrm{ext}a_\mathrm{ext})]\lesssim (\ell_\Omega [(ab),(cd)])^{-1}.\qedhere
\]
\end{proof}


\section{Applications}\label{sec:arm exponents}
\setcounter{equation}{0}

Before starting, let us mention that we will only sketch the proofs in order to highlight places which require Theorem~\ref{strong RSW}. We refer to \cite{Nol08} for complete modern proofs of the results of Sections~\ref{sub:well-separated-arms},~\ref{sub:arm-exponents} in the case of Bernoulli percolation.

\subsection{Well-separated arm events}\label{sub:well-separated-arms}

Define $\Lambda_n(x)~:=~x+[-n,n]^2$ and let $\Lambda_n=\Lambda_n(0)$. 

We begin with two classical applications of Theorem~\ref{strong RSW} (in fact, the weaker version of Theorem~\ref{strong RSW} for standard rectangles is sufficient here). The first proposition can be proved in the same way as for Bernoulli percolation, while the second is proved in \cite{DHN11}.

 \begin{proposition}\label{a priori}
For each sequence $\sigma$, there exist $\beta_\sigma,\beta'_\sigma\in (0,1)$ such that, for any $n<N$,
$$(n/N)^{\beta_\sigma}\le\phi_{\mathbb Z^2}[A_\sigma(n,N)]\le (n/N)^{\beta'_\sigma}.$$
\end{proposition}

\begin{proposition}[{\cite[Proposition~5.11]{DHN11}}]\label{mixing}
There exist $c,\alpha>0$ such that
\begin{equation*}
\big|\phi_{\mathbb Z^2}[A\cap B]-\phi_{\mathbb Z^2}[A]\phi_{\mathbb Z^2}[B]\big|\leq c\left(n/N\right)^\alpha\phi_{\mathbb Z^2}[A]\phi_{\mathbb Z^2}[B]
\end{equation*}
for any $n\leq N$ and for any event $A$ (respectively $B$) depending only on the edges in the box~$\Lambda_n$ (respectively outside $\Lambda_{2N}$).
\end{proposition}
We will also use the following fact, see the proof of \cite[Proposition~5.11]{DHN11}: up to uniform constants, the probability of any event $A$ depending only on the edges in the box~$\Lambda_N$ is independent of boundary conditions on $\partial\Lambda_{2N}$. In particular,
\begin{equation}\label{mixing2}\phi_{\mathbb Z^2}\big[A_{\sigma}(n,N)\,\big|\,\mathcal F_{\mathbb Z^2\setminus\Lambda_{2N}}\big]~\asymp~\phi_{\mathbb Z^2}[A_{\sigma}(n,N)]\quad\text{a.s.}\end{equation}
uniformly in $n$, $N$, where $\mathcal F_\Omega$ is the $\sigma$-algebra generated by (the state of) the edges in $\Omega$.
\medbreak
Let us now define the notion of {\em well-separated arms}. We refer to
Fig.~\ref{fig:arm_events} for an illustration. In what is next, let $x_k$ and $y_k$ be the endpoints\footnote{Since an arm is self-avoiding, $x_k$ and $y_k$ are uniquely defined. Furthermore, $x_k$ and $y_k$ are on the primal graph if the path is of type 1, and on the dual graph it is of type 0.} of the arm $\gamma_k$ on the inner and outer boundary respectively.
For $\delta>0$, the arms $\gamma_1,\dots,\gamma_j$ are said to be \emph{$\delta$-well-separated} if
\begin{itemize}
\item points $y_k$ are at distance larger than $2\delta N$ from each others;
\item points $x_k$ are at distance larger than $2\delta n$ from each others;
\item for every $k$, $y_k$ is $\sigma_k$-connected to distance $\delta
  N$ of $\partial\Lambda_N$ in $\Lambda_{\delta N}(y_k)$;
\item for every $k$, $x_k$ is $\sigma_k$-connected to distance $\delta
  n$ of $\partial\Lambda_n$ in $\Lambda_{\delta n}(x_k)$.
\end{itemize}
Let $A_{\sigma}^{\rm sep}(n,N)$ be the event
that $A_{\sigma}(n,N)$ occurs and there exist arms realizing
$A_{\sigma}(n,N)$ which are $\delta$-well-separated. Note that while the notation does not suggest it, this event depends on $\delta$.
 The previous definition has several convenient properties.

\begin{figure}
\begin{center}\includegraphics[width=0.44\textwidth]{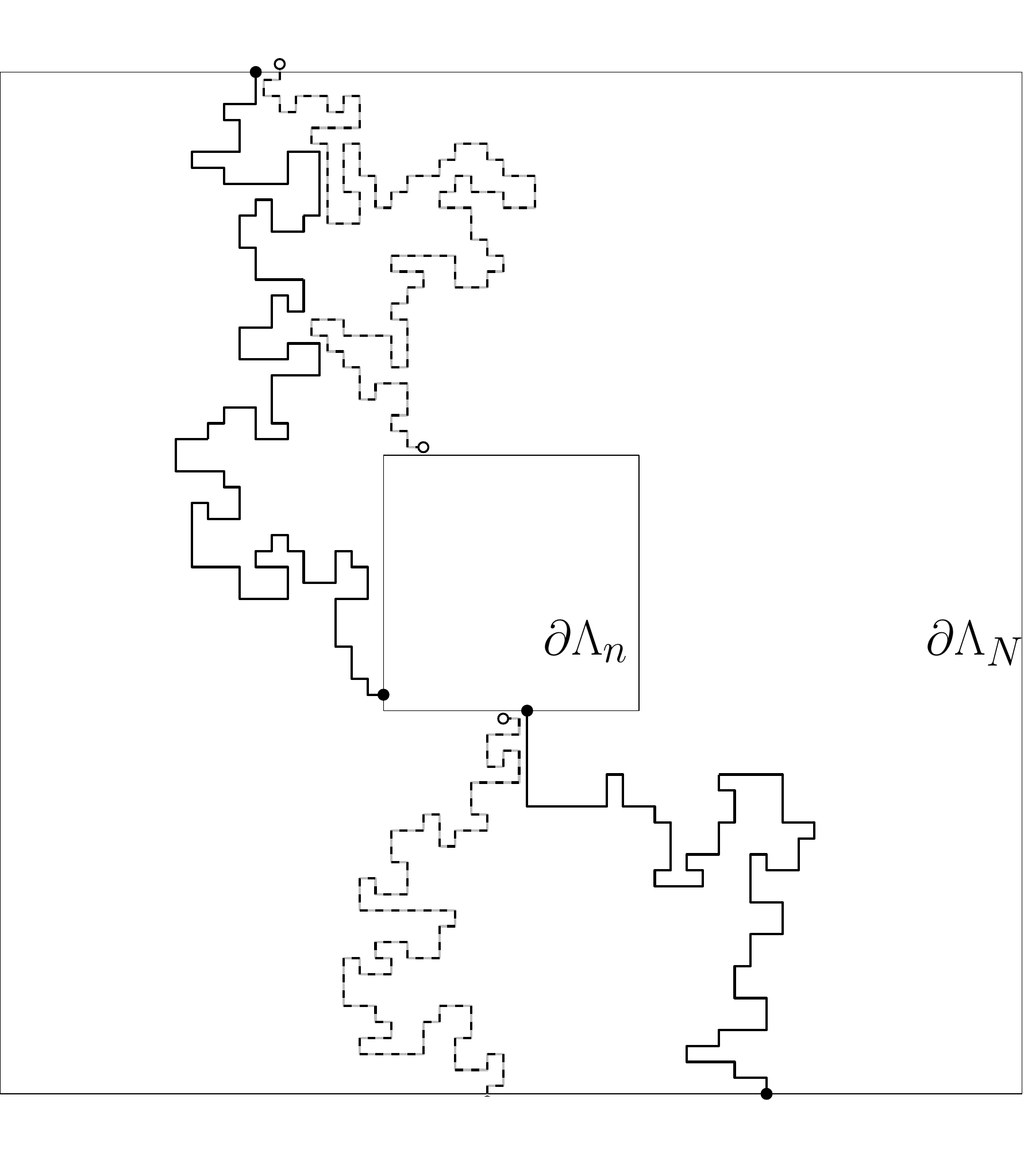}\hskip 0.07\textwidth\includegraphics[width=0.44\textwidth]{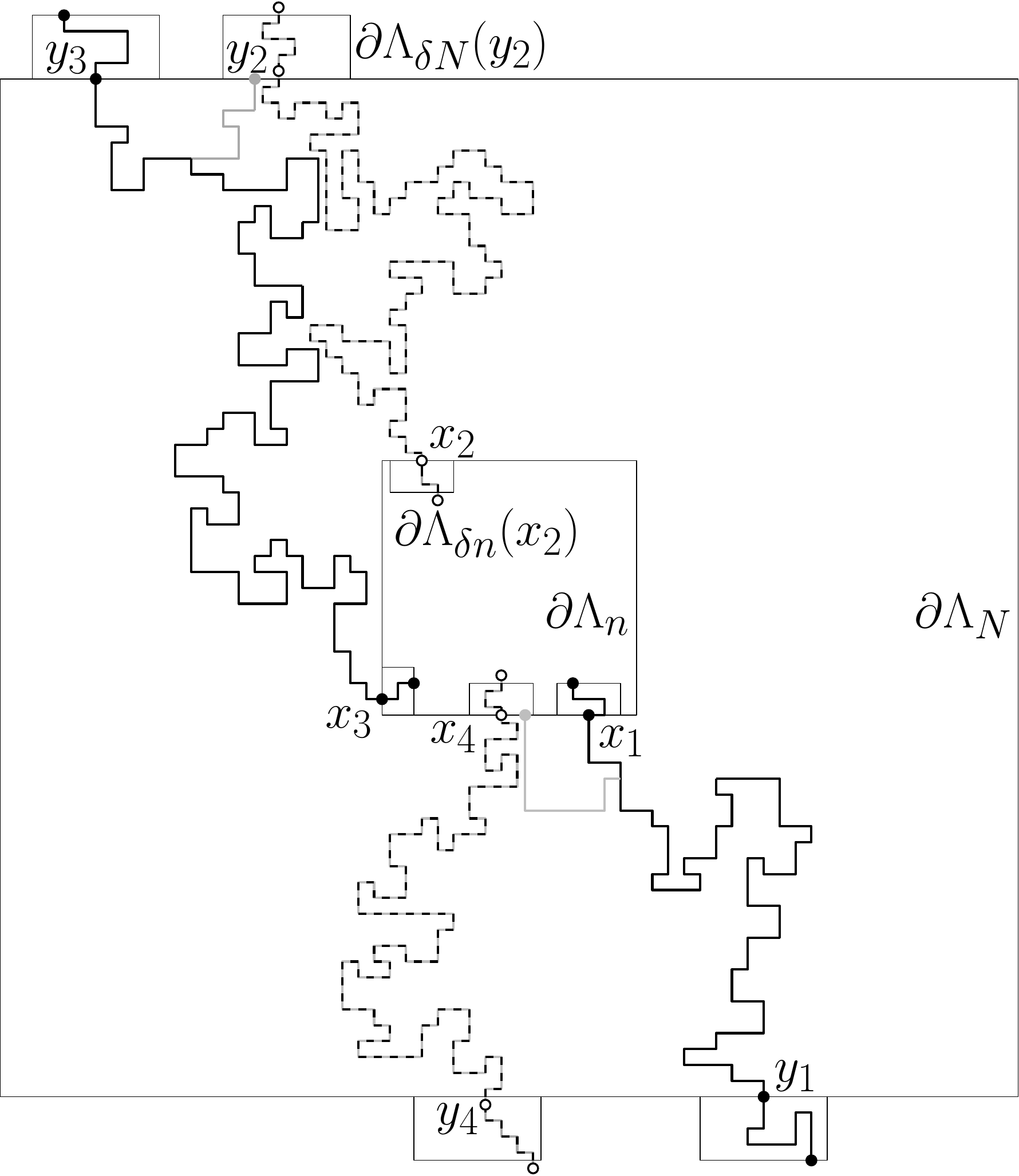}\end{center}
\caption{\label{fig:arm_events}On the left, the four-arm event $A_{1010}(n,N)$. On the right, the five-arm event $A_{1010}^{\rm sep}(n,N)$ with well-separated arms. Note that these arms are not at macroscopic distance of each other inside the domain, but only at their endpoints.}
\end{figure}

\begin{lemma}\label{extend}
  Fix a sequence $\sigma$ and let $\delta$ be small enough. For any $n_1<\frac{n_2}{2}$,
  $$\phi_{\mathbb Z^2}[A_{\sigma}^{\rm sep}(n_1,n_2)] \ges \phi_{\mathbb Z^2} [ A_{\sigma}^{\rm sep}
  (2n_1,n_2)]\,,$$
  where the constant in $\ges$ depends on $\sigma$ and $\delta$ only.
\end{lemma}

\begin{proof}
Condition on $A_{\sigma}^{\rm sep}(2n_1,n_2)$ and construct $j$ disjoint tubes of width
  $\ep=\ep(\delta)$ connecting $\Lambda_{2\delta n_1}(x_k)\setminus
  \Lambda_{2n_1}$ to disjoint boxed $\partial \Lambda_{\delta n_1}(\tilde x_k)$ for
  every $k\le j$, where $\tilde x_k\in \partial\Lambda_{n_1}$. It easily follows from topological considerations that this is possible whenever $\delta$ is small enough. Via Theorem~\ref{strong RSW}, the
  $\sigma_k$-paths connecting $x_k$ to $\partial \Lambda_{2\delta
    n_1}(x_k)\cap \Lambda_{2n_1}$ to $\partial \Lambda_{n_2}$ can be extended to connect to $\partial\Lambda_{n_1}$ while staying in tubes with positive probability
  $c=c(\sigma,\delta)$.
\end{proof}

\begin{proposition}\label{quasi_sep}
  Fix a sequence $\sigma$ and let $\delta$ be small enough. For any $n_1< n_2< \frac{n_3}{2}$,
  $$\phi_{\mathbb Z^2}[A_{\sigma}^{\rm sep}(n_1,n_3)] \ges \phi _{\mathbb Z^2}[ A_{\sigma}^{\rm sep}
  (n_1,n_2)]\cdot\phi_{\mathbb Z^2}[A_{\sigma}^{\rm sep}(n_2,n_3)]\,,$$
    where the constant in $\ges$ depends on $\sigma$ and $\delta$ only.
\end{proposition}

\begin{proof}
  We have
  \begin{align*}\phi_{\mathbb Z^2}\big[A_{\sigma}^{\rm sep}(n_1,n_2) \cap  A_{\sigma}^{\rm sep}(2n_2,n_3)
  \big] &\ = \phi_{\mathbb Z^2}\big[ A_{\sigma}^{\rm sep}(n_1,n_2)  | A_{\sigma}^{\rm sep}(2n_2,n_3)\big] \cdot \phi_{\mathbb Z^2}\big[
  A_{\sigma}^{\rm sep}(2n_2,n_3)\big]\\
 &\ges \phi_{\mathbb Z^2}\big[ A_{\sigma}^{\rm sep}(n_1,n_2)  \big] \cdot \phi_{\mathbb Z^2}\big[
  A_{\sigma}^{\rm sep}(2n_2,n_3)\big]\\
 &\ges \phi_{\mathbb Z^2}\big[ A_{\sigma}^{\rm sep}(n_1,n_2)  \big] \cdot \phi_{\mathbb Z^2}\big[
  A_{\sigma}^{\rm sep}(n_2,n_3)\big] \end{align*}
  thanks to \eqref{mixing2} and Lemma~\ref{extend}. Thus, it is sufficient to prove that
  \begin{equation} \label{x quasi sep}
  \phi_{\mathbb Z^2}\big[A_{\sigma}^{\rm sep}(n_1,n_3)\big]\ges\phi_{\mathbb Z^2}\big[A_{\sigma}^{\rm sep}(n_1,n_2) \cap A_{\sigma}^{\rm sep}(2n_2,n_3)\big].
  \end{equation}
  To do so, condition on $A_{\sigma}^{\rm sep}(n_1,n_2) \cap
  A_{\sigma}^{\rm sep}(2n_2,n_3)$ and construct $j$ disjoint tubes of width
  $\ep=\ep(\sigma,\delta)$ connecting $\Lambda_{\delta n_2}(y_k)\setminus
  \Lambda_{n_2}$ to $\Lambda_{2\delta n_2}(x_k)\cap \Lambda_{2n_2}$ for
  every $k\le j$. It easily follows from topological considerations that this is possible if $\delta$ is small enough.

  Via Theorem~\ref{strong RSW}, the arms of type
  $\sigma_k$ connecting $x_k$ to $\partial \Lambda_{2\delta
    n_2}(x_k)\cap \Lambda_{n_2}$, and $y_k$ to $\partial \Lambda_{\delta
    n_2}(y_k)\setminus \Lambda_{n_2}$ can be connected by an arm of type $\sigma_k$ staying in the corresponding tube with probability
  bounded from below by $c=c(\sigma,\delta)>0$ uniformly in everything outside these tubes, thanks to Theorem~\ref{strong RSW} (in fact, the weaker result of \cite{DHN11} would be sufficient here).
  Therefore,
  $\phi_{\mathbb Z^2}[A_{\sigma}^{\rm sep}(n_1,n_3)] \ge c\phi_{\mathbb Z^2}\big[A_{\sigma}^{\rm sep}
  (n_1,n_2) \cap A_{\sigma}^{\rm sep}(2n_2,n_3)\big]$. 
\end{proof}

\begin{remark} Proposition~\ref{quasi_sep} has the following consequence. Fix a sequence $\sigma$ and let $\delta$ be small enough. There exists
$\alpha=\alpha(\sigma,\delta)>0$ such that, for any $n_1<n_2<n_3$,
\begin{align}\label{extensibility}
  \phi_{\mathbb Z^2}\big[A_{\sigma}^{\rm sep}(n_1,n_2)\big] &\les \!
  \left({n_3}/{n_2}\right)^{\alpha} \!\cdot
  \phi_{\mathbb Z^2}\big[A_{\sigma}^{\rm sep}(n_1,n_3)\big],\\
  \phi_{\mathbb Z^2}\big[A_{\sigma}^{\rm sep}(n_2,n_3)\big] &\les \!
  \left({n_2}/{n_1}\right)^{\alpha} \!\cdot
  \phi_{\mathbb Z^2}\big[A_{\sigma}^{\rm sep}(n_1,n_3)\big].\label{extensibility inter}
\end{align}
Indeed, to prove (\ref{extensibility}),(\ref{extensibility inter}) it is sufficient to get an a priori bound $\phi_{\mathbb Z^2}[A_{\sigma}^{\rm sep}(n,N)]\ge (n/N)^{\alpha}$ for some $\alpha>0$ which can be done similarly to Proposition~\ref{a priori}.
\end{remark}

The well-separation is a powerful tool to glue arms together, but it is useful only if arms are typically well-separated. The next proposition will therefore be crucial for our study.
\begin{proposition}\label{comparison_sep_non_sep}
 Fix a sequence $\sigma$ and let $\delta$ be small enough. For any $n<N$, we have
  $\phi_{\mathbb Z^2}\big[A_{\sigma}^{\rm sep}(n,N)\big] \asymp \phi_{\mathbb Z^2}\big[A_{\sigma}(n,N)
  \big],$
  where the constants in $\asymp$ depend on $\sigma$ and $\delta$ only.
\end{proposition}

Let us start with the following two lemmas.

\begin{lemma}\label{lem:T arms}
For each $\ep>0$, there exists $T=T(\ep)>0$ such that, for any $n\ge 1$ and any boundary conditions $\xi$,
$$\phi^\xi_{\Lambda_{2n}\setminus\Lambda_n}\big[\mathit{there\ exist}~ T\mathit{\ disjoint\ arms\ crossing\ }\Lambda_{2n}\setminus\Lambda_n\big]\le \ep.$$
\end{lemma}

\begin{proof}
If $T$ arms are crossing $\Lambda_{2n}\setminus\Lambda_n$, then at least $T/4$ arms are actually crossing one of the four rectangles $[-2n,2n]\times[n,2n]$, $[-2n,2n]\times[-2n,-n]$, $[n,2n]\times[-2n,2n]$ and $[-2n,-n]\times[-2n,2n]$.
By symmetry, it is sufficient to show that, for each $\ep>0$, there exists $T>0$ such that the probability of $T$ disjoint vertical crossings of the rectangle
\[
R_n:=[-2n,2n]\times[n,2n]
\]
is bounded by $\ep$ uniformly in $n$ and boundary conditions. In fact, we only need to prove that conditionally on the existence of $k$ crossings, the probability of existence of an additional crossing is bounded from above by some constant $c<1$, since the probability of $T$ crossings is then bounded by $c^{T-1}$.

In order to prove this statement, condition on the $k$-th leftmost crossing $\gamma_k$. Assume without loss of generality that $\gamma_k$ is a dual crossing. Consider the connected component $\Omega$ of $R_n\setminus \gamma_k$ containing the right-hand side of $R_n$. The configuration in $\Omega$ is a random-cluster configuration with boundary conditions $\xi$ on $\partial R_n\cap\partial\Omega$ and free elsewhere ({i.e.} on the arc bordering the dual crossing $\gamma_k$). Now, Theorem~\ref{strong RSW} implies that $\Omega$ is crossed from left to right by a primal {\em and} a dual crossing with probability bounded from below by a universal constant. Indeed, it suffices to cut $\Omega$ into two domains $\Omega_1=\Omega\cap ([-2n,2n]\times[n,\frac{3}{2}n])$ and $\Omega_2=\Omega\cap ([-2n,2n]\times[\frac{3}{2}n,2n])$ and to assume that $\Omega_1$ is horizontally crossed and $\Omega_2$ is horizontally dual crossed. This prevents the existence of an additional vertical crossing or dual crossing of $R_n$, therefore implying the claim.
\end{proof}

\begin{remark}The previous proof harnesses Theorem~\ref{strong RSW} in a crucial way, the left boundary of $\Omega$ being possibly very rough.
Crossing estimates for standard rectangles (even with uniform boundary conditions) would not have been strong enough for this purpose. \end{remark}

Let $\delta>0$ and $n\ge1$. Define $B_n=B_n(\delta)$ to be the event that, for some sequence~$\sigma$, the annulus $\Lambda_{2n}\setminus\Lambda_n$ is crossed by disjoint arms $\gamma_1,\dots,\gamma_j$ of type $\sigma_1,\dots,\sigma_j$ but there is no $\delta$-well-separated arms $\widetilde\gamma_1,\dots,\widetilde\gamma_j$ of type $\sigma_1,\dots,\sigma_j$ such that $\widetilde\gamma_k$ is in the $\sigma_k$-cluster of $\gamma_k$ for every $k=1,\ldots,j$ ($\sigma_k$-cluster means primal cluster if $\sigma_k=1$ and dual cluster otherwise).

\begin{figure}
\begin{center} \includegraphics[width=0.80\textwidth]{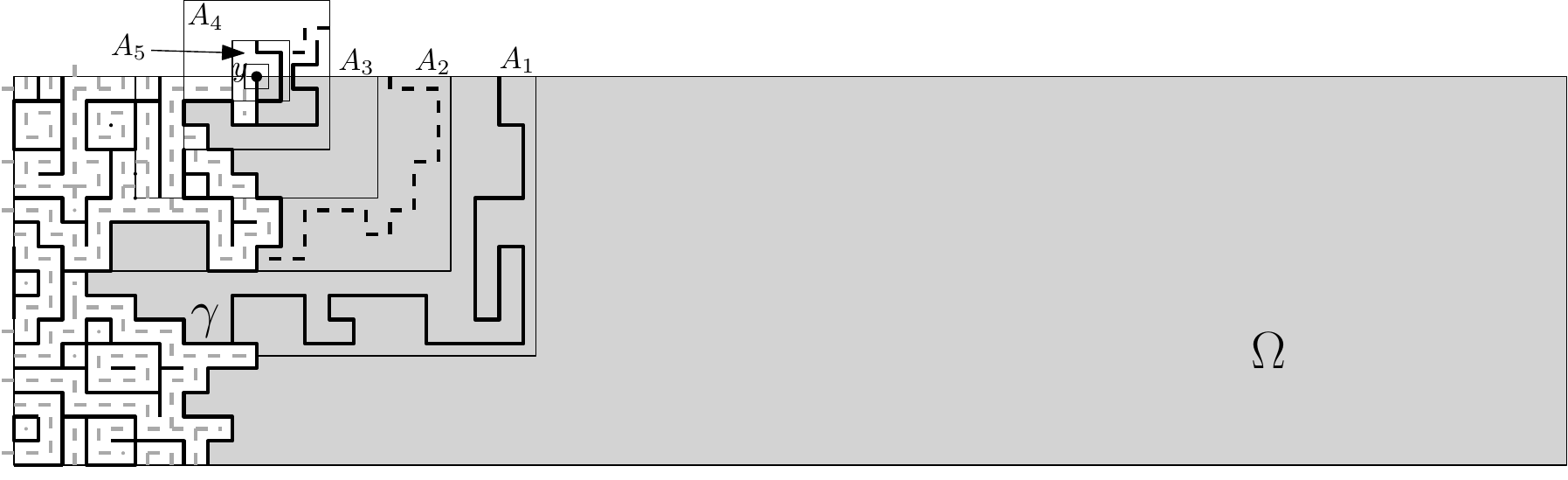}\end{center}
\caption{\label{fig:separation}The construction of open and closed paths extending the leftmost crossing $\gamma$ of $R_n$ and preventing other crossings from finishing close to $\gamma$.}
\end{figure}

\begin{lemma}\label{separation}
Let $\ep>0$. There exists $\delta>0$ such that $\phi_{\mathbb Z^2}(B_n)\le\ep$ for any $n\ge1$.
\end{lemma}

\begin{proof}
Using Lemma~\ref{lem:T arms}, consider $T$ large enough so that more than $T$ disjoint arms in $\Lambda_{2n}\setminus\Lambda_n$ exist with probability less than $\ep$. From now on, we assume that there are at most $T$ disjoint arms crossing the annulus.

Fix $\delta>0$ such that uniformly in {\em any subdomain} $D\subset \Lambda_n\setminus\Lambda_{\delta n}$ and any boundary conditions on $\partial D$, there is no
crossing from $\partial\Lambda_{\delta n}$ to $\partial\Lambda_n$ in $D$ with probability $1-\ep$ \footnote{Note that this claim is slightly stronger than simply the fact that the annulus $\Lambda_n\setminus\Lambda_{\delta n}$ is not crossed. Indeed, even if the crossing is forced to remain in $D$, the boundary conditions on $\partial D$ could help the existence of a crossing.
}. The existence of such $\delta$ can be proved easily using Theorem~\ref{strong RSW}. We may therefore assume that  no arm ends at distance less than $\delta n$ of a corner of $\Lambda_{2n}\setminus\Lambda_n$ with probability $1-8\ep$.

Similarly to the proof of Lemma~\ref{lem:T arms}, let us restrict our attention to vertical crossings in the rectangle $R_n$.
Condition on the leftmost crossing $\gamma$ and set $y$ to be the ending point of $\gamma$ on the top. Without loss of generality, let us assume that this crossing is of type 1. As before, define $\Omega$ to be the connected component of the right side of $R_n$ in $R_n\setminus \gamma$.

For $k\ge1$, let $A_k=\Lambda_{\delta^k n}(y)\setminus\Lambda_{\delta^{k+1}n}(y)$. We can assume with probability $1-\ep/T$ that no vertical crossing lands at distance $\delta^3 n$ of $y$ by making the following construction:
\begin{itemize}
\item $\Omega\cap A_1$ contains an open path disconnecting $y$ from the right-side of $R_n$;
\item $\Omega\cap A_2$ contains a dual-open path disconnecting $y$ from the right-side of $R_n$.
\end{itemize}
By choosing $\delta>0$ small enough, Theorem~\ref{strong RSW} implies that the paths in this construction exist with probability $1-\ep/T>0$ independent of the shape of $\Omega$.

For each $k\geq 4$, we may also show that $\gamma$ can be extended to the top of $A_k$ by constructing an open path in $A_k\setminus(R_n\setminus \Omega)$ from $\gamma$ to the top of $A_k$ (this occurs once again with probability $c>0$ independently of $\Omega$ and the configuration outside $A_k$), see Fig.~\ref{fig:separation}. Therefore, the probability that there exists some $k\le m$ such that this happens is larger than $1-(1-c)^{m-3}$. We find that with probability $1-\ep/T-(1-c)^{m-3}$ the path $\gamma$ can be modified into a self-avoiding crossing which is well-separated (on the outer boundary) from any crossing on the right of it by a distance at least $(\delta^3-\delta^4)n$ and that this crossing is extended to distance at least $\delta^m$ above its end-point. We may choose $m$ large enough that the previous probability is larger than $1-2\ep/T$.
One may also do the same for the inner boundary. Iterating the construction $T$ times, we find that $\phi(B_n)\le12\ep$ with $\delta^m$ as a distance of separation.
\end{proof}

\begin{proof}[Proof of Proposition~\ref{comparison_sep_non_sep}]
The lower bound $\phi_{\mathbb Z^2}[A_{\sigma}^{\rm sep}(n,N)] \le \phi_{\mathbb Z^2}[A_{\sigma}(n,N)]$ is straightforward. Let $L$ and $K$ be such that $4^{L-1}<n\le 4^L$, $4^{K+1}\le N<4^{K+2}$ and define $\widetilde B_s:=B_{2^{2s}}$. Thanks to Lemma~\ref{separation}, we fix $\delta$ small enough so that $\phi_{\mathbb{Z}^2}(\widetilde B_s)\le \ep$ for all $L\leq s\leq K$.

We may decompose the event $A_{\sigma}(n,N)$ with respect to the smallest and largest scales at which the complementary event $\widetilde B_s^c$ occurs. This gives
\begin{align*}
\phi_{\mathbb Z^2}\big[A_{\sigma}(n,N)\big]\ &\le\ \phi_{\mathbb Z^2}\biggl[\,\bigcap\nolimits_{s=L}^K \widetilde B_s\,\biggr]\\
&+\sum_{L\le \ell\le k\le K}\phi_{\mathbb Z^2}\biggl[\,\bigcap\nolimits_{s=L}^{\ell-1}\widetilde B_s\,\cap\, \big[\widetilde  B_\ell^c\cap A_{\sigma}(2^{2\ell},2^{2k+1})\cap \widetilde B_k^c\big]\,\cap\, \bigcap\nolimits_{s=k+1}^K\widetilde B_{s}\,\biggr].
\end{align*}
By definition,
\[
\widetilde B_\ell^c\cap A_{\sigma}(2^{2\ell},2^{2k+1})\cap \widetilde B_k^c~\subset~ A_{\sigma}^{\rm sep}(2^{2\ell},2^{2k+1}).
\]
Since the annuli $\Lambda_{2^{2s+1}}\setminus\Lambda_{2^{2s}}$ are separated by macroscopic areas, we can use Proposition~\ref{mixing} repeatedly to find the existence of a constant $C>0$ such that
\begin{align*}\phi_{\mathbb Z^2}[A_{\sigma}(n,N)]\ &\le\ C^{K-L}\prod\nolimits_{s=L}^K \phi_{\mathbb Z^2}[\widetilde B_s] \\
& + \sum_{L\le \ell\le k\le K}C^{K-L-(k-\ell)}\prod\nolimits_{s=L}^{\ell-1}\phi_{\mathbb Z^2}[\widetilde B_s]\cdot
\phi_{\mathbb Z^2}[A_{\sigma}^\mathrm{sep}(2^{2\ell},2^{2k+1})]\cdot
\prod\nolimits_{s=k+1}^K \phi_{\mathbb Z^2}[\widetilde B_s].\end{align*}
Recall that $\phi_{\mathbb Z^2}[\widetilde B_s]\leq \ep$ for all $s$. Furthermore, \eqref{extensibility} and \eqref{extensibility inter} show that
$$ \phi_{\mathbb Z^2}\big[A_{\sigma}^{\rm sep}(2^{2\ell},2^{2k+1})\big]\les 2^{2\alpha(\ell-L+K-k)}\phi_{\mathbb Z^2}\big[A_{\sigma}^{\rm sep}(n,N)\big]$$
for some universal constant $\alpha>0$. Altogether, we find that
\begin{align*}\phi_{\mathbb Z^2}[A_{\sigma}(n,N)]&\les \phi_{\mathbb Z^2}[A_{\sigma}^{\rm sep}(n,N)]\cdot \biggl[(C\ep2^{2\alpha})^{K-L}+\sum_{L\le \ell\le k\le K}(C\ep2^{2\alpha})^{K-L-(k-\ell)}\,\biggr] \\ & \les \phi_{\mathbb Z^2}[A_{\sigma}^{\rm sep}(n,N)].\end{align*}
provided that $\ep$ is small enough, which can be guaranteed by taking $\delta$ small enough.
\end{proof}

\subsection{Quasi-multiplicativity and universal arm exponents}\label{sub:arm-exponents}

\begin{proof}[Proof of Theorem~\ref{quasi-multiplicativity}]
If $n_2\ge\frac{n_3}{2}$, the claim is trivial. 
For $n_1<n_2<\frac{n_3}{2}$, we have
\begin{align*}
\phi_{\mathbb Z^2}\big[A_{\sigma}(n_1,n_3)\big]~&\le\phi_{\mathbb Z^2}\big[A_{\sigma}(n_1,n_2)\cap A_{\sigma}(2n_2,n_3)\big]\\
&= \phi_{\mathbb Z^2}\big[A_{\sigma}(n_1,n_2)|A_{\sigma}(2n_2,n_3)\big] \cdot\phi_{\mathbb Z^2}\big[A_{\sigma}(2n_2,n_3)\big]\\
&\asymp~\phi_{\mathbb Z^2}\big[A_{\sigma}(n_1,n_2)\big] \cdot\phi_{\mathbb Z^2}\big[A_{\sigma}(2n_2,n_3)\big]\\
&\asymp~\phi_{\mathbb Z^2}\big[A_{\sigma}(n_1,n_2)\big] \cdot\phi_{\mathbb Z^2}\big[A_{\sigma}^{\rm sep}(2n_2,n_3)\big]\\
&\les \phi_{\mathbb Z^2}\big[A_{\sigma}(n_1,n_2)\big] \cdot\phi_{\mathbb Z^2}\big[A_{\sigma}^{\rm sep}(n_2,n_3)\big]\\
&\le \phi_{\mathbb Z^2}\big[A_{\sigma}(n_1,n_2)\big]\cdot \phi_{\mathbb Z^2}\big[A_{\sigma}(n_2,n_3)\big],
\end{align*}
where we used \eqref{mixing2} in the third line, Proposition~\ref{comparison_sep_non_sep} in the fourth, and \eqref{extensibility inter} in the fifth.

On the other hand,
\begin{align*}
\phi_{\mathbb Z^2}\big[A_{\sigma}(n_1,n_3)\big]~ 
&\ges~\phi_{\mathbb Z^2}\big[A_{\sigma}^{\rm sep}(n_1,n_2)\big] \cdot\phi_{\mathbb Z^2}\big[A_{\sigma}^{\rm sep}(2n_2,n_3)\big]\\
&\asymp~\phi_{\mathbb Z^2}\big[A_{\sigma}(n_1,n_2)\big] \cdot\phi_{\mathbb Z^2}\big[A_{\sigma}(2n_2,n_3)\big]\\
&\ge~\phi_{\mathbb Z^2}\big[A_{\sigma}(n_1,n_2)\big] \cdot\phi_{\mathbb Z^2}\big[A_{\sigma}(n_2,n_3)\big].
\end{align*}
where we used (\ref{x quasi sep}) and Proposition~\ref{comparison_sep_non_sep} in the first two lines.
\end{proof}

\begin{proof}[Proof of Corollary~\ref{loc}]
The proof is classical and uses Proposition~\ref{comparison_sep_non_sep}. We refer to \cite{Nol08} for more details.
\end{proof}

\begin{figure}
\begin{center} \includegraphics[width=0.50\textwidth]{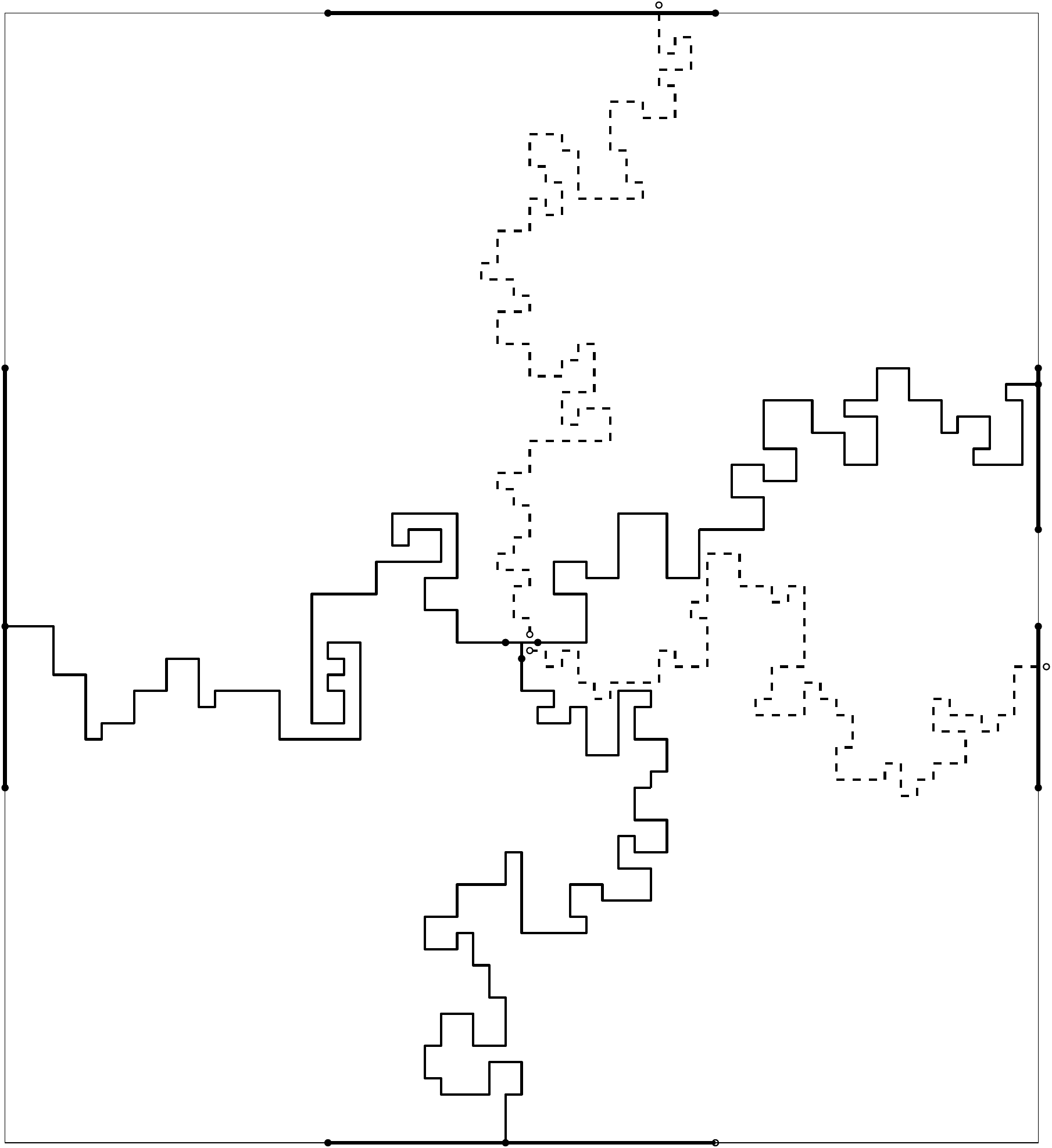}\end{center}
\caption{\label{fig:five_arm}Only one vertex per rectangle can satisfy the following topological picture.}
\end{figure}

\begin{proof}[Proof of Corollary~\ref{universal exponent}]
 We only give a sketch of the proof of the first statement; the others are derived from similar arguments (actually the arguments are slightly simpler).
By quasi-multiplicativity (Theorem~\ref{quasi-multiplicativity}), we only need to show that $ \phi_{\mathbb Z^2}\big[A_{10110}(0,N)\big] \asymp N^{-2}$.
\medbreak\noindent
 {\em Lower bound.}\ \  Fix $N>0$. Consider the
  following construction: assume that there exists a dual-open dual-path crossing
  $[-2N,2N]\times[-N,0]$ horizontally and an open path crossing
  $[-2N,2N]\times[0,N]$ horizontally. This happens with probability bounded from
  below by $c_1>0$ not depending on $N$. By conditioning on the lowest
 open self-avoiding  path $\Gamma$ crossing $[-2N,2N]\times [-N,N]$
    horizontally, the configuration in the domain $\Omega\subset\Lambda_{2N}$ above $\Gamma$ is a
  random-cluster configuration with wired boundary conditions on $\Gamma$ and undetermined boundary conditions on the other three sides (i.e. $\partial\Omega\cap\partial\Lambda_{2N}$).

Assume that (the uppermost connected component of) $\Omega\cap([-N,0]\times[-2N,2N])$ is crossed vertically by an open path, and, similarly,
$\Omega^*\cap ([\tfrac12,N-\tfrac12]\times[-2N+\tfrac12,2N-\tfrac12])$ is crossed vertically by a dual-open path. The probability of this event is once again bounded from below uniformly
  in $N$ and $\Omega$ thanks to Theorem~\ref{strong RSW}.

  {\em Here again, uniform crossing estimates for standard rectangles would not have been sufficiently strong to imply this result and Theorem~\ref{strong RSW} is absolutely necessary. }

  Summarizing, all these events occur with probability larger than
  $c_2>0$. Moreover, the existence of all these crossings implies the
  existence of a vertex in $\Lambda_{N}$ with five arms emanating from
  it, since one may observe that $\Omega\cap([-N,N]\times[-2N,2N])$ is crossed by both a primal and a dual vertical crossing, and that there exists $x$ on $\Gamma$ at the interface between two such crossings. Such an $x$ has five arms emanating from it and going to distance at least $N$\footnote{The path $\Gamma$ provides us with two primal paths going from $x$ to the boundary. Since $\Gamma$ is the lowest crossing of $[-N,0]\times[-2N,2N]$, there is an additional dual path below $\Gamma$. Finally, since $x$ is at the interface between a primal and a dual crossing above $\Gamma$, we obtain the two additional paths. Since $x$ is in $\Lambda_N$ and that arms connect $x$ and $\partial\Lambda_{2N}$, we deduce that these arms extend to distance at least $N$.}. The union bound implies $$N^2\phi_{\mathbb Z^2}[A_{10110}(0,N)]\ge c_2.$$
\medbreak\noindent
{\em Upper bound.}\ \  Recall that
  it suffices to show the upper bound for chosen
  landing sequences thanks to Corollary~\ref{loc}. Consider the event $A_x$, see Fig.~\ref{fig:five_arm}, that five mutually edge-avoiding arms $\gamma_1,\dots,\gamma_5$ of respective types 10110 are in such a way that
  \begin{itemize}
 \item $\gamma_1$ starts at $x$ and finishes on $\{N\}\times[\tfrac N4,\tfrac N2]$;
 \item $\gamma_2$ starts at $x+(\tfrac12,\tfrac12)$ and finishes on $[-\tfrac N2-\tfrac12,\tfrac N2+\tfrac12]\times\{N+\tfrac12\}$;
 \item $\gamma_3$ starts at $x$ and finishes on $\{-N\}\times[-\tfrac N2,\tfrac N2]$;
 \item $\gamma_4$ starts at $x$ and finishes on $[-\tfrac N2,\tfrac N2]\times\{-N\}$;
 \item $\gamma_5$ starts at $x+(\tfrac12,\tfrac12)$ and finishes on $\{N+\tfrac12\}\times[-\tfrac N2+\tfrac12,-\tfrac N4+\tfrac12]$.
 \end{itemize}
One may easily show  that $\phi[A_x]\asymp\phi[A_{10110}(0,N)]$ for every $x\in\Lambda_{N/2}$. In particular,
  $$N^2\phi_{\mathbb Z^2}[A_{10110}(0,N)]\asymp\sum_{x\in\Lambda_{N/2}}\phi_{\mathbb Z^2}[A_x]\le 1.$$
 The last inequality is due to the fact that the events $A_x$ are disjoint (topologically no two vertices in $\Lambda_{N}$
  can satisfy the events in question).
\end{proof}


\subsection{Spin-Ising crossing probabilities}
Recall that the FK-Ising model and the spin-Ising model are coupled, through the so-called Edwards-Sokal
coupling~\cite{ES88}. In the setup of Corollary~\ref{cor:spin-ising-crossing-bounds}, this coupling works as follows.
Let $(\Omega,a,b,c,d)$ be a topological rectangle. Consider a realization $\omega$ of the critical FK-Ising model on
$\Omega$ with boundary conditions $\xi=(ab)\cup(cd)$ (all vertices on $(bc)\cup(da)$ are wired together, all other boundary vertices are free). Let $\sigma\in\left\{\pm1\right\}^{\Omega}$ be the spin configuration obtained in the following manner:
\begin{itemize}
\item set the spins of all vertices belonging to the cluster containing $(bc)\cup(da)$ to $+1$;
\item for each of the other clusters, sample an independent fair $\pm1$ coin toss, and give that value to the spins of all vertices of this cluster.
\end{itemize}
Then $\sigma$ has the law of a critical spin-Ising configuration, with $+1$ boundary conditions on $(bc)\cup(da)$ and free boundary conditions elsewhere.

\begin{proof}[Proof of Corollary \ref{cor:spin-ising-crossing-bounds}]
For each $n_0>0$, without loss of generality, we may assume that the boundary arcs $(bc)$ and $(da)$ are distance from each other of at least $n_0$ lattice steps. Indeed, let us assume that $(bc)$ and $(da)$ are connected by a nearest-neighbor path $\gamma$ of length $n_0$. Note that the number of such paths is bounded from above by some constant $N=N(n_0,L)$ which does not depend on $(\Omega,a,b,c,d)$: if there are too many short paths connecting $(bc)$ and $(da)$, then $\ell_\Omega[(ab),(cd)]>L$. Therefore, it costs no more than some multiplicative constant (depending on $L$ and $n_0$ only) to assume that all spins along those short paths are $-1$. Let $\Omega_1,\ldots,\Omega_n$ denote the connected components of $\Omega$ appearing when all those parts are removed. By monotonicity of the spin-model with respect to boundary conditions, it is now enough to prove the claim of Corollary~\ref{cor:spin-ising-crossing-bounds} in each of $\Omega_k$ where the $+1$ boundary arcs are at least $n_0$ steps away from each other.

\smallskip

It follows from (\ref{ell self-duality}) that $\ell_\Omega[(bc),(da)]\gtrsim L^{-1}$. Provided that $n_0$ is chosen large enough, it is easy to split the topological rectangle $(\Omega,(bc),(da))$ into three \emph{connected} subdomains $(\Omega_1,(bc),(x_cx_b))$, $(\Omega_2,(x_bx_c),(x_dx_a))$, $(\Omega_3,(x_ax_d),(da))$ such that
\begin{equation}
\label{ell shares bound}
\min\{\,\ell_{\Omega_1}[(bc),(x_cx_b)],\,\ell_{\Omega_2}[(x_bx_c),(x_dx_a)],\,\ell_{\Omega_3}[(x_ax_d),(da)]\,\} \geq l(L)
\end{equation}
for some $l(L)>0$ independent of $(\Omega,a,b,c,d)$. E.g., one can use Theorem~\ref{thm:ell-Z-bounds} to get the upper bound on $\mathrm{Z}_{\Omega}[(bc),(da)]$, then apply Theorem~\ref{thm:existence-separators}(i) twice (with $k=1$), and use Theorem~\ref{thm:ell-Z-bounds} again to pass from upper bounds on the corresponding $\mathrm{Z}_{\Omega_k}$'s to (\ref{ell shares bound}). The other way to prove (\ref{ell shares bound}) (with $l(L)\asymp L^{-1}$) is to set $\Omega_k:=\{u\in\Omega: \frac{1}{3}(k\!-\!1)\leq V(u)<\frac{1}{3}k\}$, where $V$ is the electric potential in $\Omega$ (i.e. the harmonic function satisfying Neumann boundary conditions on $(ab)\cup(cd)$ and such that $V=0$ on $(bc)$, $V=1$ on $(da)$) and use definition~(\ref{ell definition}) with $g_{xy}:=|V(x)\!-\!V(y)|$ to deduce (\ref{ell shares bound}). Applying (\ref{ell self-duality}) again, we get
\[
\max\{\,\ell_{\Omega_1}[(x_bb),(cx_c)],\,\ell_{\Omega_2}[(x_ax_b),(x_cx_d)],\,\ell_{\Omega_3}[(ax_a),(x_dd)]\,\} \lesssim [l(L)]^{-1}.
\]

Now we use the Edwards-Sokal coupling between the critical spin-Ising and FK-Ising models on $\Omega$.
By Theorem \ref{thm:main-thm}(i) there exists $\alpha=\alpha(L)>0$ such that
\begin{itemize}
\item with probability at least $\alpha$ there exists \emph{no} FK open path from $(bc)$ to $(x_cx_b)$ in $\Omega_1$;
\item with probability at least $\alpha$ there \emph{exists} a FK open path from $(x_ax_b)$ to $(x_cx_d)$ in $\Omega_2$;
\item with probability at least $\alpha$ there exists \emph{no} FK open path from $(x_ax_d)$ to $(da)$ in $\Omega_3$.
\end{itemize}
So, with probability at least $\alpha^{3}$, we can guarantee that there is an FK-Ising crossing \mbox{$\gamma:\left(x_ax_b\right)\leftrightarrow\left(x_cx_d\right)$}
in $\Omega$, that does \emph{not} touch $\left(bc\right)$ and $\left(da\right)$.
Sampling a spin-Ising configuration from those FK-Ising configurations, we get that with probability at least $\frac{1}{2}\alpha^{3}$, there is
a $-1$ path from $(ab)$ to $(cd)$. Note that we need the fact that the FK~cluster of $\gamma$ is not connected to $\left(bc\right)\cup\left(da\right)$, thus its spin is defined by the fair coin toss.
\end{proof}

\begin{remark} If we consider the spin-Ising model with the following boundary conditions: $+1$ on $(bc)\cup(da)$, $-1$ on $(x_ax_b)\cup(x_cx_d)$ and free elsewhere, then, in the proof given above, it is sufficient to use the claim of Theorem~\ref{thm:main-thm} for alternating boundary conditions only. Again, by monotonicity of the spin-Ising model with respect to boundary conditions, this implies uniform bounds in terms of the discrete extremal length for the crossing probabilities in the critical spin-Ising model with ``$+1$/$-1$/$+1$/$-1$'' boundary conditions.
\end{remark}


\newcommand{\etalchar}[1]{$^{#1}$}
\providecommand{\bysame}{\leavevmode\hbox to3em{\hrulefill}\thinspace}
\providecommand{\MR}{\relax\ifhmode\unskip\space\fi MR }
\providecommand{\MRhref}[2]{%
  \href{http://www.ams.org/mathscinet-getitem?mr=#1}{#2}
}
\providecommand{\href}[2]{#2}

\end{document}